\documentclass[a4paper,10pt,reqno, english]{amsart}

\usepackage{anysize,color}
\usepackage[utf8]{inputenc}
\usepackage[dvipsnames]{xcolor}
\usepackage[colorlinks=true,urlcolor=blue,linkcolor={YellowOrange},citecolor={YellowOrange}]{hyperref}  
\usepackage{float}
\usepackage{amsfonts,amssymb,amsmath}
\usepackage{amsthm}    
\usepackage{url}
\usepackage{paralist}
\usepackage{tikz}
    \usetikzlibrary{decorations.markings,arrows,shapes.callouts}
\usepackage{rotating} 
\usepackage[mathscr]{euscript}
\usepackage{charter}
\usepackage{verbatim}
\usepackage{tikz-cd}
\usepackage{mathtools}

\newtheorem{theorem}{Theorem}[section]
\newtheorem{proposition}[theorem]{Proposition}
\newtheorem{lemma}[theorem]{Lemma} 
\newtheorem{conjecture}[theorem]{Conjecture}
\newtheorem{corollary}[theorem]{Corollary} 

\theoremstyle{definition}
\newtheorem{definition}[theorem]{Definition}
\theoremstyle{remark}

\newtheorem{remark}[theorem]{Remark}
 
\newcommand\R{\mathbb R}

\newcommand\Z{\mathbb Z}

\newcommand\sym{\mathfrak S}

\newcommand\sgn{{\rm sgn\,}}
\newcommand\ZZ{\mathscr Z}

\newcommand\K{\mathscr K}
\newcommand\B{\mathscr B}

\newcommand{\conf}{F}
\newcommand{\cellConf}{\mathscr F}

\newcommand{\wreathOfConf}[3]{C_{#1}({#2};{#3})}
\newcommand{\cellWreathOfConf}[3]{\mathscr C_{#1}({#2};{#3})}

\newcommand{\wreathOfSym}[2]{\mathscr S_{#1}({#2})}
\newcommand{\wreathOfW}[3]{W_{#1}({#2};{#3})}

\newcommand{\littleCubeOperad}[2]{\mathrm{Cubes}_{#1}({#2})}

\newcommand\perim{\mathrm{perim}}
\newcommand\leb{{\mathscr L}^d}
\newcommand\res{\mathrm{res}}
\newcommand\trf{\mathrm{trf}}

\DeclareMathOperator{\EMP}{\mathrm{EMP}}
\DeclareMathOperator{\EAP}{\mathrm{EAP}}
\DeclareMathOperator{\aut}{\mathrm{Aut}}
\DeclareMathOperator{\interior}{\mathrm{int}}
\DeclareMathOperator{\id}{\mathrm{id}}
\DeclareMathOperator{\dH}{\textrm{d}_{\textrm H}}
\DeclareMathOperator{\dS}{\textrm{d}_{\textrm S}}
\DeclareMathOperator{\dist}{\textrm{dist}}
\DeclareMathOperator{\conv}{\textrm{conv}}
\DeclareMathOperator{\orient}{\mathrm{orient}}

\begin{document}

\title{Convex Equipartitions inspired by the little cubes operad
}
\dedicatory{Dedicated to G\"unter M. Ziegler on the occasion of his 60th Birthday}

\author[Blagojevi\'c]{Pavle V. M. Blagojevi\'{c}} 
\thanks{The research by Pavle V. M. Blagojevi\'{c} leading to these results has
        received funding from the Serbian Ministry of Science, Technological development and Innovations. The research of Nikola Sadovek is funded by the Deutsche Forschungsgemeinschaft (DFG, German Research Foundation) under Germany's Excellence Strategy – The Berlin Mathematics Research Center MATH+ (EXC-2046/1, project ID 390685689, BMS Stipend).}
\address{Inst. Math., FU Berlin, Arnimallee 2, 14195 Berlin, Germany\hfill\break
\mbox{\hspace{4mm}}Mat. Institut SANU, Knez Mihailova 36, 11001 Beograd, Serbia}
\email{blagojevic@math.fu-berlin.de} 
\author[Sadovek]{Nikola Sadovek} 
\address{Inst. Math., FU Berlin, Arnimallee 2, 14195 Berlin, Germany}
\email{nikolasdvk@gmail.com}

\begin{abstract}
 
A decade ago two groups of authors,  Karasev, Hubard \& Aronov and Blagojevi\'c \& Ziegler, have shown that the regular convex partitions of a Euclidean space into $n$ parts yield a solution to the generalised Nandakumar \& Ramana-Rao conjecture when $n$ is a prime power. 
This was obtained by parametrising the space of regular equipartitions of a given convex body with the classical configuration space. 

Now, we repeat the process of regular convex equipartitions many times,  first partitioning the Euclidean space into $n_1$ parts, then each part into $n_2$ parts,  and so on. 
In this way we obtain iterated convex equipartions of a given convex body into $n=n_1\cdots n_k$ parts. 
Such iterated partitions are parametrised by the (wreath) product of classical configuration spaces.
We develop a new configuration space -- test map scheme for solving the generalised Nandakumar \& Ramana-Rao conjecture using the Hausdorff metric on the space of iterated convex equipartions.

The new scheme yields a solution to the conjecture if and only if all the $n_i$'s are powers of the same prime. 
In particular, for the failure of the scheme outside prime power case we give three different proofs.

\end{abstract}

\maketitle

\medskip
\section{Introduction and statement of the main result}
\medskip

In Nandakumar's blog entry \cite{Nandakumar06} from 2006, Nandakumar and Ramana-Rao asked whether every convex polygon in the plane can be partitioned into any prescribed number $n$ of convex pieces that have equal area and equal perimeter.
They \cite{NandakumarRamanaRao12} gave an answer in case $n=2$ using the intermediate value theorem and proposed a proof for the case $n=2^k\geq 4$.
B\'ar\'any,  Blagojevi\'c \& Sz\H ucs \cite[Thm.\,1.1]{BaranyBlagojevicSzuecs10} gave the positive answer for the case $n=3$ by setting an appropriate configuration space -- test map (CS--TM) scheme which allowed efficient use of Fadell--Husseini ideal valued index theory. 
In 2018, the question was proved to be true by Akopyan, Avvakumov \& Karasev \cite{AkopyanAvvakumovKarasev18}.

\medskip
\begin{theorem}
\label{theorem: nandakumar and ramana rao d=2}
	 Every convex polygon $P$ in the plane can be partitioned into any prescribed number $n$ of convex pieces that have equal area and equal perimeter.
\end{theorem}

The question whether a higher dimensional extension of Theorem \ref{theorem: nandakumar and ramana rao d=2} holds was formulated in  \cite[Thm.\,1.3]{KarasevHubardAronov14} and \cite[Thm.\,1.3]{BlagojevicZiegler15}. Sober\'on \cite[Thm.\,1]{Soberon12} solved a similar question, namely the B\'ar\'any's conjecture of equally partitioning $d$ measures in $\R^d$ into $k$ pieces, using clever modification of the configuration space which allowed use of Dold's theorem \cite{Dold83}.

\medskip
\begin{conjecture} [Generalised Nandakumar \& Ramana-Rao] 
\label{conjecture: nandakumar and ramana rao for general d}
Let $K$ be a $d$-dimensional convex body in $\R^d$, let $\mu$ be an absolutely continuous probability measure on $\R^d$, let $n \ge 2$ be any natural number, and let $\varphi_1, \dots ,\varphi_{d-1}$ be any $d-1$ continuous functions on the metric space of $d$-dimensional convex bodies in $\R^d$. 
Then there exists a partition of $\R^d$ into $n$ convex pieces $P_1, \dots ,P_n$ such that equalities
	\[
		\mu(P_1 \cap K) = \dots = \mu(P_n \cap K)
	\]
	and
	\[
		\varphi_i(P_1 \cap K) = \dots = \varphi_i(P_n \cap K)
	\]
hold for every $1 \le i \le d-1$.
\end{conjecture} 

\medskip
In the previous statement, absolute continuity of probability measure in $\R^d$ is meant with respect to the Lebesgue measure on $\R^d$, and the set of $d$-dimensional convex bodies in $\R^d$ is endowed with the Hausdorff metric.

\medskip
Karasev, Hubard \& Aronov \cite{KarasevHubardAronov14} proposed a solution for the generalised Nandakumar \& Ramana-Rao problem using a CS--TM scheme based on the generalised Voronoi diagrams. 
The scheme was used to show that Conjecture \ref{conjecture: nandakumar and ramana rao for general d} follows from the non-existence of an $\sym_n$-equivariant map
\begin{equation} \label{eq: map from configuration space to the sphere}
	\conf(\R^d,n) \longrightarrow S(W_n^{\oplus d-1}),
\end{equation}
where the symmetric group $\sym_n$ acts on $\conf(\R^d,n) = \{(x_1, \dots, x_n) \in (\R^d)^{ n}: x_i \neq x_j \text{ for } i \neq j\}$ and $W_n = \{x\in \R^n:x_1+ \dots + x_n = 0\}$ by permuting the points and coordinates, respectively. 
Using a homological analogue of obstruction theory, they showed \cite[Thm.\,1.10]{KarasevHubardAronov14} non-existence of the map \eqref{eq: map from configuration space to the sphere} when $n$ is a prime power, and so gave the positive answer to the conjecture in this case. 

\medskip    
Blagojevi\'c \& Ziegler \cite[Thm.\,1.2]{BlagojevicZiegler15} showed that a map \eqref{eq: map from configuration space to the sphere} does not exist if and only if $n$ is a prime power using equivariant obstruction theory. 
In order to apply the method, they developed a Salvetti-type CW-model $\cellConf(d,n)$ of the configuration space $\conf(\R^d, n)$, which is moreover equivariant deformation retract. 
Resistance of Conjecture \ref{conjecture: nandakumar and ramana rao for general d} to the existing topological methods is manifested in the fact that the map \eqref{eq: map from configuration space to the sphere} exists whenever $n$ is not a prime power.

\medskip
Blagojevi\'c, L\"uck \& Ziegler \cite[Thm.\,8.3]{BlagojevicLueckZiegler15} showed the non-existence of a map \eqref{eq: map from configuration space to the sphere} for $n$ a prime using an analogue of Cohen's vanishing theorem \cite[Thm.\,8.2]{Cohen76}.

\medskip
In this paper we develop a cohomological method for identifying new classes of solutions to the generalised Nandakumar \& Ramana-Rao conjecture. 
The method stems from a geometric idea of \emph{iterated} generalised Voronoi diagrams. 
Namely, let $n = n_1 \cdots n_k$ be a multiplicative decomposition of the total number of parts in which we want to partition the convex body. 
In the first iteration the convex body is partitioned into $n_1$ convex pieces of equal area, and inductively, for each $2 \le i \le k$, in the $i$th iteration one further divides each of the $n_1\cdots n_{i-1}$ pieces into $n_i$ convex pieces of equal area. 
We call such partitions {\em iterated of level $k$ and type $(n_1, \dots ,n_k)$}.

Our method is developed in three steps. We begin by fixing a $d$-dimensional convex body $K$ in $\R^d$.

\medskip
In the first step, presented in Section \ref{section: cs--tm scheme}, we define a space $\wreathOfConf{k}{d}{n_1, \dots , n_k}$, called {\em the wreath product of configuration spaces}, which is used to parametrise iterated partitions of level $k$ and type $(n_1, \dots ,n_k)$ of $K$. 
Namely, we produce a map
\begin{equation} 
	\label{eq: parametrisation map, intro}
	\wreathOfConf{k}{d}{n_1, \dots , n_k} ~ \longrightarrow ~ \EMP_{\mu}(K, n),
\end{equation}
where the codomain is the (metric) space of equal mass partitions of $K$ into $n$ pieces with respect to some absolutely continuous measure $\mu$. 
Moreover, our setting has an appropriate symmetry of the iterated semi-direct product $\wreathOfSym{k}{n_1, \dots ,n_k}$ of the symmetric groups, and the map \eqref{eq: parametrisation map, intro} is equivariant with respect to it.
Now, by setting an appropriate CS--TM scheme and using the parametrisation map \eqref{eq: parametrisation map, intro}, we show that the existence of an iterated solution of level $k$ and type $(n_1, \dots ,n_k)$ to the general Nandakumar \& Ramana-Rao problem follows from non-existence of an $\wreathOfSym{k}{n_1, \dots ,n_k}$-equivariant map of the form
\begin{equation*} 
	\wreathOfConf{k}{d}{n_1, \dots ,n_k} ~ \longrightarrow ~ S(\wreathOfW{k}{d-1}{n_1, \dots ,n_k}).
\end{equation*}
Here $\wreathOfW{k}{d-1}{n_1, \dots ,n_k}$ stands for the relevant $\wreathOfSym{k}{n_1, \dots ,n_k}$-representation, and the corresponding unit sphere is denoted by $S(\wreathOfW{k}{d-1}{n_1, \dots ,n_k})$.

\medskip
In the second step, which is ``outsourced''  to Section \ref{section: continuity of partitions}, we prove the result on continuity of partitions. 
Namely, the CS--TM scheme from \cite{KarasevHubardAronov14} and \cite{BlagojevicZiegler15}  used the fact that for a $d$-dimensional convex body $K \subseteq \R^d$ there exists a continuous $\sym_n$-equivariant map
\[
	\conf(\R^d, n) ~ \longrightarrow ~ \EMP(K,n), \qquad x ~\longmapsto ~ (K_1(x), \dots, K_n(x)),
\]
where $(K_1(x), \dots, K_n(x))$ denotes the convex partition of $K$ into $n$ parts of equal measure obtained from the generalised Voronoi diagram with sites associated to $x \in \conf(\R^d, n)$. 
The existence of such a map follows from the theory of optimal transport. 
We generalise this fact and prove that the map
\[
	\K^d \times \conf(\R^d, n) ~ \longrightarrow ~ (\K^d)^{\times n}, \qquad(K, x) ~ \longmapsto ~ (K_1(x), \dots, K_n(x)) \in \EMP(K,n)
\]
is continuous, where $\K^d$ denotes the space of all $d$-dimensional convex bodies in $\R^d$ endowed with the Hausdorff metric. 
Continuity of this map, stated in Theorem \ref{theorem: continuity of partitions}, is a crucial result which justifies the continuity of the parametrisation map \eqref{eq: parametrisation map, intro}, and therefore verifies the validity of the CS--TM scheme from Section \ref{subsection: cs-tm scheme}.

\medskip
In the final and third step, contained in Section \ref{section: equivariant obstruction theory}, we prove the main topological result of the paper using equivariant obstruction theory.

\medskip
\begin{theorem} \label{theorem: main}
	Let $k \ge 1$, $d \ge 2$, and $n_1, \dots ,n_k \ge 2$ be integers. Then, an $\wreathOfSym{k}{n_1, \dots ,n_k}$-equivariant map of the form 
	\begin{equation} \label{eq: wreath product map in the introduction}
		\wreathOfConf{k}{d}{n_1, \dots ,n_k} \longrightarrow  S(\wreathOfW{k}{d-1}{n_1, \dots ,n_k})
	\end{equation}
	does not exist if and only if $n_1, \dots ,n_k$ are powers of the same prime number.
\end{theorem} 

\medskip
 The proof outlined in Section \ref{section: equivariant obstruction theory} is of the type "if and only if" and gives a complete answer to the question of effectiveness of the introduced CS--TM scheme in the following sense:
 \begin{compactitem}[\quad --]
 	\item If all the numbers $n_1, \dots ,n_k$ are powers of the same prime number, then an $\wreathOfSym{k}{n_1, \dots ,n_k}$-equivariant map of the form \eqref{eq: wreath product map in the introduction} does not exist. Consequently, the CS--TM scheme guaranties existence of an iterated solution of Conjecture \ref{conjecture: nandakumar and ramana rao for general d}.
 	
 	\item If $n_1, \dots ,n_k$ are not all powers of the same prime number, then there exists an $\wreathOfSym{k}{n_1, \dots ,n_k}$-equivariant map of the form \eqref{eq: wreath product map in the introduction}, hence the CS--TM scheme is not able to give any insights into the existence of a solution to Conjecture \ref{conjecture: nandakumar and ramana rao for general d}.
 \end{compactitem}

\medskip
Finally, as a corollary of this approach, we obtain the top equivariant cohomology of the wreath product of configuration spaces with twisted integral coefficients.
 
\medskip
In addition, we include one more proof of the non-existence of the equivariant map in Theorem \ref{theorem: main} and three more proof of their existence, in respective cases. 
We thank the anonymous referee for pointing out two of these proofs.

\medskip
More precisely, in Section \ref{section: equivariant obstruction theory} we reduce the question of non-existence of the map \eqref{eq: wreath product map in the introduction} the non-existence result from the PhD thesis of Pali\'c \cite{Palic18}, which were obtained in collaboration with Blagojevi\'c \& Karasev. 
In Section \ref{section: existence via operad} we prove the existence of an equivariant map \eqref{eq: wreath product map in the introduction} using the structural map of the little cubes operad when $n_1, \dots, n_k$ are not powers of the same prime number. At the end of the section, we produce another argument, similar in spirit, by invoking the existence result of Avvakumov, Karasev \& Skopenkov \cite{AvvakumovKarasevSkopenkov23} for $d \ge 3$, and the existence result of Avvakumov \& Kudrya \cite{AvvakumovKudrya21} when $d=2$ and $n_1 \cdots n_k$ is moreover not twice the prime power.
For further highlights see the recent work of Michael Crabb \cite{crabb2022multivalued}.
In Section \ref{section: existence via ozaydin} we prove existence of the same map by appropriately adapting in detail the trick of \"Ozaydin in \cite{Ozaydin87}.

\medskip
A direct consequence of Theorem \ref{theorem: main} is a new positive solution of the generalised Nandakumar \& Ramana-Rao conjecture.

\medskip
\begin{corollary}[Iterated solution to generalised Nandakumar \& Ramana-Rao]
	Let $K$ be a $d$-dimensional convex body in $\R^d$, let $\mu$ be an absolutely continuous probability measure on $\R^d$, let $n_1, \dots , n_k \ge 2$ be integers, and let $\varphi_1, \dots ,\varphi_{d-1}$ be any $d-1$ continuous functions on the metric space of $d$-dimensional convex bodies in $\R^d$. 
	If $n_1, \dots , n_k$ are all powers of the same prime number, then there exists an iterated partition of $K$ of level $k$ and type $(n_1, \dots, n_k)$ into $n := n_1 \cdots n_k$ convex pieces $K_1, \dots ,K_n$ such that equalities
	\[
		\mu(K_1) = \dots = \mu(K_n)
	\]
	and
	\[
		\varphi_i(K_1) = \dots = \varphi_i(K_n)
	\]
	hold for every $1 \le i \le d-1$.
\end{corollary}

\medskip
Moreover, the proof method, as noted by Karasev \cite[Thm.\,1.6]{Karasev10}, implies also the following result,

\medskip
\begin{corollary}
	Let $K$ be a $d$-dimensional convex body in $\R^d$, $n \ge 2$ an integer,  $n = n_1 \cdots n_k$ a multiplicative decomposition where $n_1, \dots , n_k$ are prime powers, $\mu$ an absolutely continuous probability measure on $\R^d$, and let $\varphi_1, \dots ,\varphi_{d-1}$ be any $d-1$ additive, continuous functions on the metric space of $d$-dimensional convex bodies in $\R^d$. Then there exists an iterated partition of $K$ of level $k$ and type $(n_1, \dots, n_k)$ into $n$ convex pieces $K_1, \dots ,K_n$ such that equalities
	\[
		\mu(K_1) = \dots = \mu(K_n)
	\]
	and
	\[
		\varphi_i(K_1) = \dots = \varphi_i(K_n)
	\]
	hold for every $1 \le i \le d-1$.
\end{corollary}

\medskip
The idea of iterated partitions seems to have appeared first in the context of Gromov's waist of the sphere theorem \cite{Gromov03}, \cite{Memarian11}. 
There, Gromov \cite[Thm.\,4.4.A]{Gromov03} considered partitions of the sphere into $2^i$ pieces, which were parametrised by the wreath products of spheres. 
In order to obtain a slightly different waist of the sphere result, Pali\'c \cite[Thm.\,5.2.5]{Palic18} considered partitions of the sphere into $p^i$ pieces indexed by the $i$th wreath product of the configuration spaces on $p$ points. 
Iterated partitions of Euclidean spaces appeared in the work of Blagojevi\'c \& Sober\'on \cite[Sec.\,2]{BlagojevicSoberon18}, where they were parametrised by the $i$th join of the configuration space.
Moreover, Akopyan, Avvakumov \& Karasev \cite{AkopyanAvvakumovKarasev18} used the method of iterated convex partitions of a convex body into prime number of pieces in every step to prove Theorem \ref{theorem: nandakumar and ramana rao d=2}.
For the purposes of this paper we use a more general version of the wreath product of configuration spaces, in the sense that the number of points of the configuration space in different levels do not need to coincide.

\medskip
\subsection*{Acknowledgements.}
The authors are grateful to the referee of the paper for careful reading and many insightful suggestions which guided our revision and improved the quality of the manuscript significantly.
In addition, Pavle Blagojevi\'c would like to thank Roman Karasev for many intriguing discussion over the years and many relevant comments and suggestions on the current work.

\medskip
\section{The Configuration Space -- Test Map scheme} 
\label{section: cs--tm scheme}
\medskip
In this section we define the spaces needed for developing the configuration space -- test map scheme. The main result of the section is Theorem \ref{theorem: cs-tm scheme} where we show that the existence of an iterated \mbox{solution} to the generalized Nandakumar \& Ramana Rao conjecture follows from the non-existence of a certain \mbox{equivariant} map.

\medskip
Let us denote by $\K^d$ the space of all full dimensional convex bodies in the Euclidean space $\R^d$ endowed with the Hausdorff metric and let $K \in \K^d$. 
As in the statement of Conjecture \ref{conjecture: nandakumar and ramana rao for general d}, let $\mu$ be a probability measure on $\R^d$ which is absolutely continuous with respect to the Lebesgue measure. 

\medskip
For an integer $n \ge 2$ and a convex body $K \in \K^d$, let us say that $(K_1, \dots ,K_n) \in (\K^d)^{\times n}$ form a {\em convex partition of $K$ into $n$ pieces} if
\begin{compactitem}[\quad --]
	\item $K = K_1 \cup \dots \cup K_n$, and
	\item $\interior(K_i) \cap \interior(K_j) = \emptyset$ for each $1 \le i < j \le n$.
\end{compactitem}
Note that, implicitly, $\interior(K_i)\neq \emptyset$, for all $1 \le i  \le n$, because we require that $K_i\in \K^d$.

\medskip
\begin{definition}
	Let $n \ge 2$ be an integer, let $K \in \K^d$ be a convex body, and let $\mu$ be a probability measure on $\R^d$ which is absolutely continuous with respect to the Lebesgue measure. We define $\EMP_{\mu}(K ,n)$ to be the space of all convex partitions of $K$ into $n$ pieces $(K_1, \dots ,K_n) \in \K^d$ such that 
	\[
		\mu(K_1) = \dots = \mu(K_n),
	\]
	endowed with the product Hausdorff metric induced from $(\K^d)^{\times n}$.
\end{definition} 

\medskip
In the case of the plane, that is $d = 2$, the perimeter function induces an $\sym_n$-equivariant map
\begin{equation*}
	\EMP_{\mu}(K,n) \longrightarrow W_n
\end{equation*}
which maps an equal mass partition $(K_1, \dots ,K_n)$ to the vector obtained from $(\perim(K_1), \dots ,\perim(K_n))$ by \mbox{subtracting} the average sum of all coordinates from each individual coordinate. 
Thus, an equal area partition $(K_1, \dots ,K_n) \in \EAP(K,n)$ of $K$ gives a solution to the Nandakumar \& Ramana-Rao conjecture if and only if it is mapped to the origin by the above map. 

\medskip
Analogously, in the general case $d \ge 2$, there exists a map
\begin{equation} \label{eq:EMP_mu to Wn^d-1}
	\EMP_{\mu}(K,n) \longrightarrow W_n^{\oplus d-1}
\end{equation}
induced by the functions $\varphi_1, \dots , \varphi_{d-1}\colon \K^d \longrightarrow \R$ from the statement of Conjecture \ref{conjecture: nandakumar and ramana rao for general d}, which hits the origin in $W_n^{\oplus d-1}$ if and only if there is a solution to the generalised Nandakumar \& Ramana-Rao conjecture. 
See \cite[Sec.\,2]{KarasevHubardAronov14} and \cite[Sec.\,2]{BlagojevicZiegler15} for more details.

\medskip

One of the ways in which the existence of zeros of the map \eqref{eq:EMP_mu to Wn^d-1} has been approached so far was by restricting the attention to a subspace of the domain of the, so called, {\em regular partitions} of $K$. These partitions are the ones which arise from piecewise-linear convex functions.

\medskip
Namely, using the the theory of optimal transport (see \cite[Sec.\,2]{KarasevHubardAronov14} and \cite[Sec.\,2]{BlagojevicZiegler15} for more details), one can parametrise regular convex partitions $(K_1, \dots ,K_n)$ of $K$ into pieces of equal mass by the classical configuration space $F(\R^d,n)$ of $n$ pairwise distinct points in $\R^d$. 
For an illustration see Figure \ref{fig 01}.
More precisely, there exists an $\sym_n$-equivariant map
\begin{equation} \label{eq:F to EMP_mu}
	\conf(\R^d, n) \longrightarrow \EMP_{\mu}(K,n).
\end{equation}
One can think of the parametrisation \eqref{eq:F to EMP_mu} as a prescribed way to assign to any $n$ pairwise distinct points in the plane, called {\em the sites}, a partition of  $K$ into $n$ convex pieces of equal mass. 
In particular, if the image of the composition
\begin{equation} \label{eq: cs-tm scheme, general d}
    \conf(\R^d,n) \longrightarrow \EMP_{\mu}(K,n) \longrightarrow W_n^{\oplus d-1}
\end{equation}
of maps \eqref{eq:F to EMP_mu} and \eqref{eq:EMP_mu to Wn^d-1} contains the origin, there exists a {\em regular} solution to Theorem \ref{theorem: nandakumar and ramana rao d=2}.

\begin{figure}[h]
\includegraphics[scale=1]{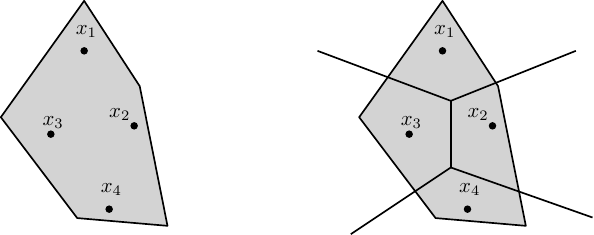}
\caption{\small A point $(x_1,x_2,x_3,x_4)\in \conf(\R^2,4)$ induces a regular partition $(K_1,K_2,K_3,K_4)\in \EMP(K,4)$ of the convex body $K$.}
\label{fig 01}	
\end{figure}

\medskip
In \cite[Thm.\,1.10]{KarasevHubardAronov14} it was shown that whenever $n$ is a power of prime, every $\sym_n$-equivariant map of the form 
\[
    \conf(\R^d,n) \longrightarrow W_n^{\oplus d-1}
\]
hits the origin, thus showing the existence of a solution to the (generalised) Nandakumar \& Ramana-Rao conjecture in this case which is regular. Moreover, in \cite[Thm.\,1.2]{BlagojevicZiegler15} it was shown that outside of the prime power case, there always exists an $\sym_n$-equivariant map $\conf(\R^d,n) \longrightarrow W_n^{\oplus d-1} \setminus \{0\}$. 
This last result used equivariant obstruction theory, and it showed the limits of the above configuration space -- test map scheme proposed by the map \eqref{eq: cs-tm scheme, general d}.

\medskip
Known topological methods are able to give solutions to the Nandakumar \& Ramana-Rao conjecture in the form of regular partitions. 
Since the space $\EMP_{\mu}(K,n)$ of convex equipartitions of $K$ is larger than the set of regular partitions of $K$, a natural question arises: {\em Are there solutions to the Nandakumar \& Ramana-Rao conjecture which are not regular?}

\medskip
In this section we develop a method to parametrise iterated partitions of a given full-dimensional convex body in $\R^d$ into $n$ parts. 
Namely, we parametrise partitions with $k$ iteration by a wreath product of configuration spaces $C_k(d;n_1,..,n_k)$, where $n = n_1 \cdots n_k$ and $n_1, \dots ,n_k \ge 2$.

\medskip
The rest of the section is organised as follows. 
In Section \ref{subsection: iterated partitions: first example} we give an example of the parametrisation in the case when $d=2$ and $k = 2$. 
Then, in Section \ref{subsection: partitions types}, we describe partitions of level $k$ and type $(n_1, \dots ,n_k)$.
 Parametrisation of such partitions by the space $C_k(d;n_1,..,n_k)$ is presented in Section \ref{subsection: cs-tm scheme}, where the configuration space -- test map scheme is set up.

\subsection{Iterated partitions: the first example} \label{subsection: iterated partitions: first example}

We establish a new configuration space -- test map scheme with the intention of identifying a wider class of solutions to Conjecture \ref{conjecture: nandakumar and ramana rao for general d}. Let us first consider the case $d=2$. 
To this end, we parametrise a class of possibly non-regular equal area partitions of $K$ into $n$ convex pieces. 
In this section, we give an example of such a parametrisation for iterated partition of level two in the plane.

\medskip
Let $a \ge 2$ and $b \ge 2$ be natural numbers such that $n=ab$. 
Consider a parametrisation of equal area partitions of $K$ into convex pieces
\begin{equation} \label{eq:F^b x F to EAP, d=2}
	\conf(\R^2, a)^{\times b} \times \conf(\R^2, b) \longrightarrow \EAP(K,n)
\end{equation}
obtained as the composition of the following two maps:
\begin{compactenum} [\quad --]
\item The first map
\[
	\conf(\R^2, a)^{\times b} \times \conf(\R^2, b) \longrightarrow \conf(\R^2, a)^{\times b} \times \EAP(K,b)
\]
is the the product of the identity on $\conf(\R^d, a)^{\times b}$ and the parametrisation \eqref{eq:F to EMP_mu} on $F(\R^2, b)$.

\item The second map
\[
	\conf(\R^2, a)^{\times b} \times \EAP(K,b) \longrightarrow \EAP(K,n)
\]
on each slice $\conf(\R^d, a)^{\times b} \times \{(K_1, \dots ,K_b)\}$ of the domain equals to the product of maps
\[
	\conf(\R^2, a) \longrightarrow\EAP(K_i,a)
\]
for $1 \le i \le b$, followed by the inclusion $\EAP(K_1,a)\times \dots \times \EAP(K_b,a) \subseteq \EAP(K,n)$. 
In Section \ref{section: continuity of partitions} we discuss in more details the continuity of this map.
\end{compactenum}
The map \eqref{eq:F^b x F to EAP, d=2} gives a parametrisation of equal area partition of $K$ into $n$ convex pieces which arise in two iterations. 
Namely, in the first iteration we divide $K$ into $b$ convex pieces $(K_1, \dots ,K_b)$ by a point in $\conf(\R^2, b)$ according to the rule \eqref{eq:F to EMP_mu}, and in the second iteration we again use the rule \eqref{eq:F to EMP_mu} to divide each of the pieces $K_1, \dots, K_b$ by $b$ points in $\conf(\R^2,a)$. 
This is an example of an {\em iterated partition of level two and type (a,b)}.
For an illustration see Figure \ref{fig 00}.

\begin{figure}[h]
\includegraphics[scale=1.2]{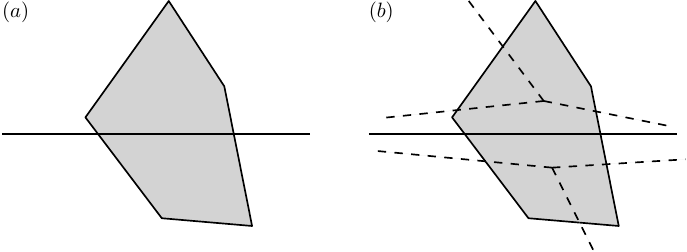}
\caption{\small Partition of convex body of:  $(a)$ level $1$ and type $2$, $(b)$ level $2$ and type $(3,2)$.}
\label{fig 00}	
\end{figure} 

\medskip
In Section \ref{subsection: iterated partitions: first example} we present a more general approach and define {\em iterated partition of level $k$ and type $(n_1, \dots ,n_k)$}, for any multiplicative decomposition $n = n_1 \cdots n_k$ with $n_1, \dots ,n_k \ge 2$. 
The idea is that in the first iteration one divides $K$ into $n_1$ convex pieces of equal area using a point in $\conf(\R^2, n_1)$ according to the rule \eqref{eq:F to EMP_mu}. 
Continuing inductively, for each $2 \le i \le k$, in the $i$th iteration one further divides each of the $n_1\cdots n_{i-1}$ pieces into $n_i$ convex pieces of equal area by as many points in $\conf(\R^2, n_i)$, again according to the rule \eqref{eq:F to EMP_mu}.

\medskip
Let us consider the level two example in the plane once again and discuss the symmetries of the parameter space. 
In this context, the natural group  acting on the domain $\conf(\R^2, a)^{\times b} \times \conf(\R^2, b)$ of the map \eqref{eq:F^b x F to EAP, d=2} is the semi-direct product $(\sym_a)^{\times b} \rtimes \sym_b$, where $\sym_b$ acts on the product $(\sym_a)^{\times b}$ by permuting the factors. 
More precisely, the action is given by
\begin{equation} \label{eq: action of semi-direct product, level two}
	(\tau_1, \dots ,\tau_b; \sigma) \cdot (y_1, \dots , y_b; x) = (\tau_1 \cdot y_{\sigma^{-1}(1)}, \dots ,\tau_a \cdot y_{\sigma^{-1}(b)}; \sigma \cdot x),
\end{equation}
for $(\tau_1, \dots ,\tau_b; \sigma) \in (\sym_a)^{\times b} \rtimes \sym_b$ and $(y_1, \dots , y_b; x) \in \conf(\R^2, a)^{\times b} \times \conf(\R^2, b)$. The action of the symmetric group on the configuration space is assumed to permute the coordinates.

The action \eqref{eq: action of semi-direct product, level two} induces an action on  the codomain of the parametrisation map \eqref{eq:F^b x F to EAP, d=2} in such a way that it becomes $((\sym_a)^{\times b}\rtimes \sym_b)$-equivariant. 
Indeed, an element $(\tau_1, \dots ,\tau_b; \sigma) \in (\sym_a)^{\times b} \rtimes \sym_b$ acts on a partition 
$
	(K_{i,1}, \dots ,K_{i,a})_{i = 1}^b \in \EAP(K,n)
$
by
\[
	(\tau_1, \dots ,\tau_b; \sigma) \cdot (K_{i,1}, \dots ,K_{i,a})_{i = 1}^b = (K_{\sigma^{-1}(i),\tau^{-1}(1)}, \dots ,K_{\sigma^{-1}(i),\tau^{-1}(a)})_{i = 1}^b.
\]
Thus, the full $\sym_n$-symmetry of the partition is broken. 

\medskip
Let us now define a map
\begin{equation} \label{eq: perimeter-induced map, d=2}
	\EAP(K,n) \longrightarrow (W_a)^{\oplus b} \oplus W_b
\end{equation}
which sends a partition
\[
	(K_{i,1}, \dots ,K_{i,a})_{i = 1}^b \in \EAP(K,n)
\]
to a vector which:
\begin{compactitem}[\quad --]
	\item for each $1 \le i \le b$, on the coordinates of the $i$th copy of $W_a$ has the vector obtained from
		\[
			(\perim(K_{i,1}), \dots ,\perim(K_{i,a})) \in \R^a
		\]
		by subtracting the average of all $a$ coordinates;
	\item on the coordinates of $W_b$ has the vector 
		\[
			\Big(\perim(K_{i,1}) + \dots + \perim(K_{i,a})-\frac1b\sum_{1\leq i\leq b} \big(\perim(K_{i,1}) + \dots + \perim(K_{i,a})\big)\Big)_{i=1}^b \in W_b	.	
		\]
\end{compactitem}
In analogy to \eqref{eq: action of semi-direct product, level two}, the actions of $\sym_b$ and $\sym_a$ on $W_b$ and $W_a$, respectively, induce an action of the semi-direct product $(\sym_a)^{\times b}\rtimes \sym_b$ on the vector space $(W_a)^{\oplus b} \oplus W_b$, making the map \eqref{eq: perimeter-induced map, d=2} equivariant with respect to it.

\medskip
Consequently, the $((\sym_a)^{\times b}\rtimes \sym_b)$-equivariant composition
\[
	\conf(\R^2, a)^{\times b} \times \conf(\R^2, b) \longrightarrow \EAP(P,n) \longrightarrow (W_a)^{\oplus b} \oplus W_b
\]
of \eqref{eq:F^b x F to EAP, d=2} and \eqref{eq: perimeter-induced map, d=2} hits the origin, now in $(W_a)^{\oplus b} \oplus W_b$, if and only if there exists an iterated solution of level two and type $(a,b)$. In Theorem \ref{theorem: main} we prove that any $((\sym_a)^{\times b}\rtimes \sym_b)$-equivariant map of the form 
 \[
	\conf(\R^2, a)^{\times b} \times \conf(\R^2, b) \longrightarrow (W_a)^{\oplus b} \oplus W_b
\]
hits the origin if and only if $a$ and $b$ are powers of the same prime number. Thus, the configuration space -- test map scheme is able to give a positive iterated solution of level two and type $(a,b)$ to the Nandakumar \& Ramana-Rao conjecture only in the case when $a$ and $b$ are powers of the same prime.

\subsection{Partition types} \label{subsection: partitions types}

In this section, for integers $k \ge 1$ and $n_1, \dots ,n_k \ge 2$, we define iterated partitions of a given convex body $K \in \K^d$ which are of {\em level $k$ and type $(n_1, \dots ,n_k)$}. 
Furthermore we identify a natural groups of symmetries on such partition spaces.

\medskip
For each convex body $K \in \K^d$ and each integer $n \ge 1$, there exists an $\sym_n$-equivariant map 
\begin{equation} \label{eq:F to EMP}
	\conf(\R^d, n) \longrightarrow \EMP(K,n)
\end{equation}
parametrising regular equal mass partitions of $K$ into $n$ convex pieces. Here $\EMP(K,n) \subseteq (\K^d)^{\times n}$ denotes the metric space of all equal mass partitions of $K$ into $n$ convex pieces. See \cite[Sec.\,2]{KarasevHubardAronov14} or \mbox{\cite[Sec.\,2]{BlagojevicZiegler15}} for more details about this map.

\medskip
As in Section \ref{subsection: iterated partitions: first example}, map \eqref{eq:F to EMP} can be thought of as a prescribed way to divide $K$ into $n$ convex parts of equal mass using $n$ pairwise distinct points in $\R^d$. One can now repeat the process for each of the $n$ parts individually, and divide them regularly into $m \ge 2$ parts, obtaining a partition of $P$ into $nm$ parts, which is not necessarily regular. 
Continuing this process, each of the $nm$ parts can be divided further, etc. 
In this sense, regular partitions \eqref{eq:F to EMP} can be thought of as partitions arising from a single iteration, and in the following definition we introduce a notion of a partition coming from arbitrarily many iterations.
An illustration of the iteration process see Figure \ref{fig 02}.

\medskip
\begin{definition}
	Let $K \in \K^d$ be a convex body.
\begin{compactitem}[\quad --]
	\item (Iterated partition of level one) Let $n \ge 2$ be an integer. Any partition of $K$ in the image of the map \eqref{eq:F to EMP} is said to be \emph{iterated of level $1$ and type $n$}.
	
	\item (Iterated partition of level $k \ge 2$) Suppose $k \ge 2$ and $n_1, \dots ,n_k \ge 2$ are integers. Let
	\[
		(K_1, \dots , K_{n_1 \cdots n_{k-1}}) \in \EMP(K, n_1 \cdots n_{k-1}),
	\]
	be a partition of $K$ of level $k-1$ and type $(n_1, \dots ,n_{k-1})$, and let $(K_{i,1}, \dots ,K_{i,n_k})$, for  $1 \le i \le n_1 \cdots n_{k-1}$, be any partition of $K_i$ of level $1$ and type $n_k$. Then the partition 
	\[
		(K_{i,j}: 1 \le i \le n_1 \cdots n_{k-1},~ 1 \le j \le n_k)
	\]
	of $K$ into $n_1 \cdots n_k$ convex parts of equal mass is said to be \emph{iterated of level $k$ and type $(n_1, \dots ,n_{k})$}.
\end{compactitem}
\end{definition}

\medskip
An iterated partition of level $k$ and type $(n_1, \dots ,n_k)$ arises in $k$ inductive steps. For every iteration $0 \le i \le k-1$, each of the $n_k \cdots n_{k-i+1}$ elements of the partition in the $i$th step contains exactly $n_{k-i}$ elements in the partition from the $(i+1)$st step. Therefore, the collection of all such intermediate parts forms a ranked poset with respect to inclusion. More precisely:
\begin{compactitem}[\quad --]
	\item The maximum of the poset is the full convex body $K$.
	
	\item For each $1 \le i \le k$ the elements of the poset which are of level $i$ are precisely the elements of the partition which arise in the $i$th inductive step.
\end{compactitem}
For fixed parameters $k \ge 1$ and $n_1, \dots ,n_k \ge 2$, the induced posets arising from different convex bodies are isomorphic. Since the posets are supposed to model the type of iterated division process, from now on we will fix one poset of each type.

\begin{figure}[h]
\includegraphics[scale=0.85]{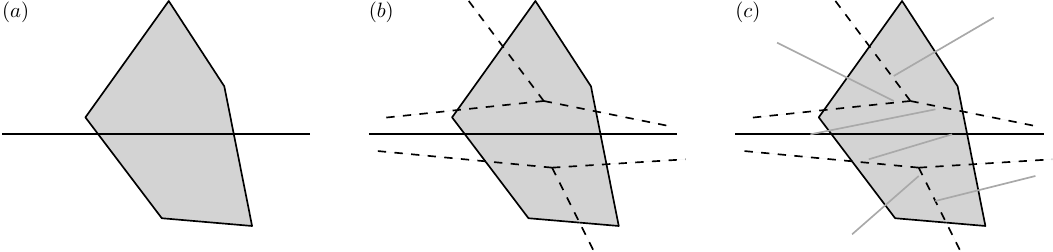}
\caption{\small Partition of convex body of:  $(a)$ level $1$ and type $2$, $(b)$ level $2$ and type $(3,2)$, level three and type $(2,3,2)$.}
\label{fig 02}	
\end{figure} 

\medskip
\begin{definition}
    Let $k \ge 1$ and $n_1, \dots ,n_k \ge 2$ be integers. Let us denote by $P_k(n_1, \dots ,n_k)$ a fixed poset, called {\em poset of level $k$ and type $(n_1, \dots ,n_k)$}, built inductively as follows.
    \begin{compactitem}[\quad --]
        \item (Case $k=1$) Let $P_1(n_1)$ be a poset on $n_1+1$ elements $\pi^0,\pi_1^1, \dots ,\pi_{n_1}^1$ with relations 
        \[
            \pi_1^1, \dots ,\pi_{n_k}^1 \preccurlyeq \pi^0.
        \]
        The maximum $\pi^0$ is at the {\em level zero} and elements $\pi_1^1, \dots ,\pi_{n_1}^1$ are at the {\em level one}.

        \item (Case $k \ge 2$) Let the poset $P_k(n_1, \dots ,n_k)$ be obtained from the poset $P_{k-1}(n_2, \dots ,n_{k})$ by adding $n_1 \cdots n_k$ elements $\pi_1^k, \dots ,\pi_{n_1\cdots n_k}^k$ of {\em level $k$} and relations
        \[
            \pi_{(i-1)n_k + 1}^k, \dots , \pi_{(i-1)n_k + n_k}^k \preccurlyeq \pi^{k-1}_i
        \]
        for each $1 \le i \le n_1 \cdots n_{k-1}$.
    \end{compactitem}
\end{definition}

\medskip
Notice that the Hasse diagram of the poset $P_k(n_1,...,n_k)$ is a rooted tree with $k$ levels, where a vertex of level $0 \le i \le k-1$ has precisely $n_{k-i}$ ``children''.
For an example of posets  $P_k(n_1,...,n_k)$ see Figure \ref{fig 03}.

\begin{figure}[h]
\includegraphics[scale=0.75]{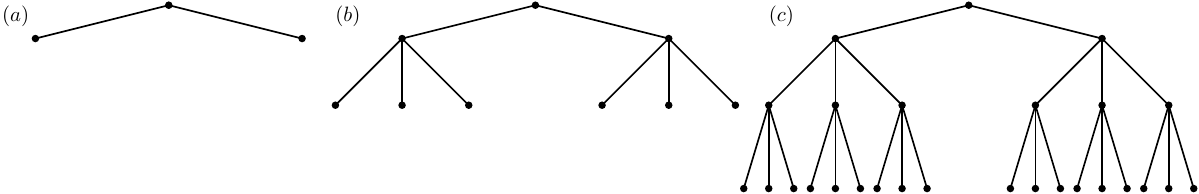}
\caption{\small Hasse diagram of posets:  $(a)$ $P_1(2)$, $(b)$ $P_2(3,2)$, $(c)$ $P_3(3,3,2)$.}
\label{fig 03}	
\end{figure} 

\medskip
For integers $k \ge 1$ and $n_1, \dots ,n_k \ge 2$, let us denote by $\aut_k(n_1, \dots ,n_k)$ the group of automorphisms of the poset $P_k(n_1, \dots ,n_k)$. 

\medskip
In the case $k=1$, it is useful to consider iterated partitions of level $1$ and type $n$ together with the group of symmetries of the poset $P_1(n)$. 
In fact, the group of symmetries in this case is $\aut_1(n) \cong \sym_n$ which coincides with the natural group of symmetries of the configuration space $\conf(\R^d, n)$.
This group was essentially used for identification of partitions of level $1$ and type $n$ which solve the general Nandakumar \& Ramana-Rao conjecture \cite{KarasevHubardAronov14}, \cite{BlagojevicZiegler15}.

\medskip
More generally, the group $\aut_k(n_1, \dots ,n_k)$ naturally acts on the set of partitions of level $k$ and type $(n_1, \dots ,n_k)$, and will be crucial ingredient for the configuration map -- test map scheme set up in Section \ref{subsection: cs-tm scheme} to identify iterated solutions to the general Nandakumar \& Ramana-Rao conjecture which are of level $k$ and type $(n_1, \dots ,n_k)$.

\medskip
\begin{definition}
	Let $k\ge 1$ and $n_1, \dots ,n_k \ge 2$ be integers. We define the group $\wreathOfSym{k}{n_1, \dots ,n_k}$ inductively as follows.
	\begin{compactitem}[\quad --]
		\item For $k=1$ we set $\wreathOfSym{1}{n_1} := \sym_{n_1}$.
		
		\item For $k \ge 2$ let
		\[
			\wreathOfSym{k}{n_1, \dots ,n_k} := \wreathOfSym{k-1}{n_1, \dots ,n_{k-1}}^{\times n_k} \rtimes \sym_{n_k}
		\]
		be the semi-direct product; for a definition of semi-direct product consult for example \cite[Sec.\,5.5]{DummitFoote04}. The symmetric group $\sym_{n_k}$ acts on the product $\wreathOfSym{k-1}{n_1, \dots ,n_{k-1}}^{\times n_k}$ by
		\[
		\sigma \cdot (\Sigma_1, \dots ,\Sigma_{n_k}) = (\Sigma_{\sigma^{-1}(1)}, \dots ,\Sigma_{\sigma^{-1}(n_k)}),
		\]
		for $\sigma \in \sym_n$ and $(\Sigma_1, \dots ,\Sigma_{n_k}) \in \wreathOfSym{k-1}{n_1, \dots ,n_{k-1}}^{\times n_k}$.
	\end{compactitem}
\end{definition}

\medskip
Written down in more details for $k \ge 2$, the operation in the wreath product $\wreathOfSym{k}{n_1, \dots ,n_k}$ is given inductively by
\[
	(\Sigma_1, \dots ,\Sigma_{n_k};\sigma) \cdot (\Theta_1, \dots ,\Theta_{n_k}; \theta) = (\Sigma_1 \cdot \Theta_{\sigma^{-1}(1)}, \dots ,\Sigma_{n_k} \cdot \Theta_{\sigma^{-1}(n_k)}; \sigma \theta),
\]
where $(\Sigma_1, \dots ,\Sigma_{n_k};\sigma), (\Theta_1, \dots ,\Theta_{n_k}; \theta) \in  \wreathOfSym{k-1}{n_1, \dots ,n_{k-1}}^{\times n_k} \rtimes \sym_{n_k} = \wreathOfSym{k}{n_1, \dots ,n_k}$.

\medskip
\begin{lemma} \label{lem: aut = wreathSym}
	Let $k\ge 1$ and $n_1, \dots ,n_k \ge 2$ be integers. Then, there is a group isomorphism
	\[
		\aut_k(n_1, \dots ,n_k) ~\cong~ \wreathOfSym{k}{n_1, \dots ,n_k}.
	\]
\end{lemma}
\begin{proof}
	The proof is by induction on $k \ge 1$, because the definitions of relevant objects were given inductively. 
	
	\medskip
	For $k = 1$ we have that $\aut_1(n_1) \cong \sym_{n_1}$. 
	Indeed, $\sigma \in \sym_{n_1}$ acts on the level one elements of $P_1(n_1)$ by the rule $\pi_i^1 \longmapsto \pi_{\sigma(i)}^1$ for each $1 \le i \le n$.
	
	\medskip
	Suppose now that $k \ge 2$. 
	By the induction hypothesis, it is enough to show the existence of an isomorphism
	\begin{equation} \label{eq: inductive iso for aut}
		\aut_{k-1}(n_1, \dots ,n_{k-1})^{\times n_k} \rtimes \sym_n \xrightarrow{~\cong~} \aut_k(n_1, \dots ,n_k).
	\end{equation}
	One way to construct an isomorphism is to map an element 
	\[
		(\phi_1, \dots ,\phi_n; \sigma) \in \aut_{k-1}(n_1, \dots ,n_{k-1})^{\times n_k} \rtimes \sym_n
	\]
	by a composition of the following two maps.
	\begin{compactenum}[\quad (1)]
        \item The first one permutes the level one elements $\pi^1_i$, together with their respective principal order ideals $\langle \pi^1_i \rangle := P_k(n_1, \dots ,n_k)_{\preccurlyeq \pi_i^1} \cong P_{k-1}(n_1, \dots , n_{k-1})$, by the rules		
        \[
			\pi^1_i \longmapsto \pi^1_{\sigma(i)} ~~ \text{ and } ~~ \langle \pi^1_i \rangle \xrightarrow{~ \cong ~} \langle \pi_{\sigma(i)}^1 \rangle,
		\] 
		for each $1 \le i \le n_k$, where the (order) ideals are mapped in the canonical way preserving the order of the indices on each level.
		\item The second one applies the automorphism $\phi_i$ to the (order) ideal $\langle \pi_i^1 \rangle$ via the unique isomorphism $\langle \pi_i^1 \rangle \cong P_{k-1}(n_1, \dots ,n_{k-1})$ which preserves the order of the indices, for each $1 \le i \le n_k$.
	\end{compactenum}
	The map \eqref{eq: inductive iso for aut} is a bijection, since each element of $\aut_k(n_1, \dots ,n_k)$ is charaterised by the permutation of its level-one elements and the automorphism of their respective principle (order) ideals.
	 An illustration of the isomorphism \eqref{eq: inductive iso for aut} is presented Figure \ref{fig 04}.
\end{proof}

\begin{figure}[h]
\includegraphics[scale=0.77]{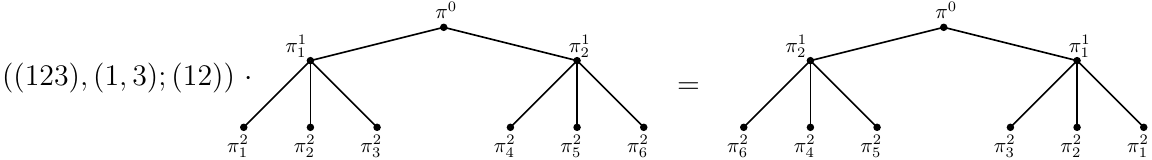}
\caption{\small Action of the element $((123),(13);(12))\in\sym_3$.}
\label{fig 04}	
\end{figure} 

\medskip
\begin{definition} \label{def: wreath product action}
	Assume that a group $G$ acts on the set $X$ and that the symmetric group $\sym_{n}$ acts on the set $Y$. 
	Let $G^{\times n} \rtimes \sym_n$ be the semi-direct product, where $\sym_n$ acts on the product $G^{\times n}$ by permuting the coordinates. The induced action of $G^{\times n} \rtimes \sym_n$ on the set $X^{\times n} \times Y$, given by
	\[
		(g_1, \dots , g_n; \sigma) \cdot (x_1, \dots , x_n; y) := (g_1 \cdot x_{\sigma^{-1}(1)}, \dots , g_n \cdot x_{\sigma^{-1}(n)}; \sigma \cdot y)
	\]
	for each $(g_1, \dots , g_n; \sigma) \in G^{\times n} \rtimes \sym_n$ and $(x_1, \dots , x_n; y) \in X^{\times n} \times Y$, is called {\em the wreath product action}.
\end{definition}

\subsection{Configuration Space -- Test Map scheme} \label{subsection: cs-tm scheme}

In the previous section we described partitions of $K$ into $n := n_1 \cdots n_k$ parts of level $k$ and type $(n_1, \dots ,n_k)$. 
Hence, all such partitions are contained in the space $\EMP(K, n)$, but is it possible to continuously parametrise them by a topological space? 

\medskip
In this section we first give a positive answer to the previous question, and then develop the related configuration space -- test map scheme for finding  solutions to Conjecture \ref{conjecture: nandakumar and ramana rao for general d} which are of level $k$ and type $(n_1, \dots ,n_k)$.

\medskip
\begin{definition} [Wreath product of Configuration Spaces]
Let $d, k\ge 1$ and $n_1, \dots ,n_k \ge 2$ be integers. 
The space $\wreathOfConf{k}{d}{n_1, \dots ,n_k}$ is defined inductively as follows.
	\begin{compactitem}[\quad --]
		\item For $k=1$ set $\wreathOfConf{1}{d}{n_1} := \conf(\R^d, n_1)$.
		
		\item For $k \ge 2$ let
		\[
			\wreathOfConf{k}{d}{n_1, \dots ,n_k} := \wreathOfConf{k-1}{d}{n_1, \dots ,n_{k-1}}^{\times n_k} \times \conf(\R^d, n_k).
		\]
	\end{compactitem}
\end{definition}

\medskip
The action of the symmetric group $\sym_n$ on the configuration space $\conf(\R^d, n)$ induces an action of the group $\wreathOfSym{k}{n_1, \dots ,n_k}$ on $\wreathOfConf{k}{d}{n_1, \dots ,n_k}$. 
Let us introduce this action explicitly by induction on $k \ge 1$. 
\begin{compactitem}[\quad --]
		\item For $k=1$ the group $\wreathOfSym{1}{n} = \sym_{n}$ acts on $\wreathOfConf{1}{d}{n} = \conf(\R^d, n)$ by permuting the coordinates
		\[
			\sigma \cdot (x_1, \dots, x_{n}) = (x_{\sigma^{-1}(1)}, \dots, x_{\sigma^{-1}(n)}),
		\]
		where $\sigma \in \sym_n$ and $(x_1,\dots,x_n) \in \conf(\R^d, n)$.
		\item For $k \ge 2$ and $n_1, \dots ,n_k \ge 2$ the group $\wreathOfSym{k}{n_1, \dots ,n_k}$ acts on the space $\wreathOfConf{k}{d}{n_1, \dots ,n_k}$ by the rule
		\[
			(\Sigma_1, \dots ,\Sigma_{n_k};\sigma) \cdot (X_1, \dots ,X_{n_k}; x) = (\Sigma_1 \cdot X_{\sigma^{-1}(1)}, \dots ,\Sigma_{n_k} \cdot X_{\sigma^{-1}(n_k)}; \sigma \cdot x),
		\]
		where
		\[
			(\Sigma_1, \dots ,\Sigma_{n_k};\sigma) \in \wreathOfSym{k-1}{n_1, \dots ,n_{k-1}}^{\times n_k} \rtimes \sym_{n_k} = \wreathOfSym{k}{n_1, \dots ,n_k}
		\]
		and
		\[
			(X_1, \dots ,X_{n_k}; x) \in \wreathOfConf{k-1}{d}{n_1, \dots ,n_{k-1}}^{\times n_k} \times \conf(\R^d, n_k) = \wreathOfConf{k}{d}{n_1, \dots ,n_k}.
		\]
\end{compactitem}
In similar situations, where the action is defined analogously, we will call it simply {\em the wreath product action}, in the light of Definition \ref{def: wreath product action}.

\medskip
Next, for a given convex body $K \in \K^d$, we define a parametrisation of the iterated equal mass convex  partitions of $K$.

\begin{proposition}
	[Parametrisation of iterated partitions] \label{proposition: parametrisation of iterated partitions}
	Let $k \ge 1$, $d\ge 1$ and $n_1, \dots ,n_k \ge 2$ be integers, and let $K \in \K^d$ be a $d$-dimensional convex body. Then, there exists a continuous $\wreathOfSym{k}{n_1, \dots ,n_k}$-equivariant map
	\begin{equation*}
		\wreathOfConf{k}{d}{n_1, \dots ,n_k} \longrightarrow \EMP_{\mu}(K, n_1 \cdots n_k).
	\end{equation*}
\end{proposition} 
\begin{proof}
	As already discussed in the case when $\mu$ is the Lebesgue measure itself, for each integer $n \ge 2$ and $K \in \K^d$, due to the theory of optimal transport (see \cite[Sec.\,2]{KarasevHubardAronov14} and \cite[Sec.\,2]{BlagojevicZiegler15}) there exists an $\sym_n$-equivariant map
	\begin{equation} \label{eq: slice-wise partition map}
		\conf(\R^d, n) \times \{K\} \longrightarrow \EMP_{\mu}(K, n), ~~~ (x, K) \longmapsto (K_1(x), \dots ,K_n(x)),
	\end{equation}
	where $(K_1(x), \dots ,K_n(x))$ represents the unique regular equipartition of $K$ with respect to the measure $\mu$ and sites $x = (x_1, \dots ,x_n) \in \conf(\R^d, n)$. In Lemma \ref{lemma: existence of unique weights} we provide a topological proof of the existence of the map \eqref{eq: slice-wise partition map}.
	
	\medskip
	One can consider \eqref{eq: slice-wise partition map} as the map given on slices $\conf(\R^d, n) \times \{K\}\subseteq \conf(\R^d, n) \times \K^d$ and construct an $\sym_n$-equivariant map with enlarged domain
	\begin{equation} \label{eq: partition map with sites}
		p: \conf(\R^d, n) \times \K^d \longrightarrow (\K^d)^{\times n}, ~~~~~ (x, K) \longmapsto (K_1(x), \dots ,K_n(x)) \in \EMP_{\mu}(K, n),
	\end{equation}
	where the symmetric group acts on the $\K^d$-coordinate of the domain trivially, and on the codomain by permuting the coordinates. Continuity of the map \eqref{eq: partition map with sites} is proved in Theorem \ref{theorem: continuity of partitions}, and it, in particular, implies continuity of the map \eqref{eq: slice-wise partition map}.
	
	\medskip
	Assuming the existence of a map \eqref{eq: partition map with sites}, we prove something a bit stronger than the statement of the proposition. Namely, we show existence of an $\wreathOfSym{k}{n_1, \dots ,n_k}$-equivariant map
	\begin{equation} \label{eq: p_k}
		p_k: \wreathOfConf{k}{d}{n_1, \dots ,n_k} \times \K^d \longrightarrow (\K^d)^{\times n_1 \cdots n_k}
	\end{equation}
	which on each slice 
	\[
		  \wreathOfConf{k}{d}{n_1, \dots ,n_k} \times \{K\} \subseteq \wreathOfConf{k}{d}{n_1, \dots ,n_k} \times \K^d
	\]
	equals to the desired $\wreathOfSym{k}{n_1, \dots ,n_k}$-equivariant map
	\begin{equation*}
	 \wreathOfConf{k}{d}{n_1, \dots ,n_k} \times \{K\} \longrightarrow \EMP_{\mu}(K, n_1 \cdots n_k) \subseteq (\K^d)^{\times n_1 \cdots n_k}
	\end{equation*}
	from the statement of this proposition. 
	The $\wreathOfSym{k}{n_1, \dots ,n_k}$-action on the $\K^d$-coordinate of the domain of the map \eqref{eq: p_k} is trivial, and the action on the codomain is given inductively by the wreath product action (see Definition \ref{def: wreath product action}).
	
	\medskip
	The construction of a map $p_k$ is done by induction on $k \ge 1$. For the base case $k = 1$, the map \eqref{eq: partition map with sites} has the desired restriction properties \eqref{eq: slice-wise partition map}, hence we set $p_1 := p$. 
	Let now $k \ge 2$ and set ${\bf n} := (n_1, \dots , n_k)$ and ${\bf n}' := (n_1, \dots ,n_{k-1})$ to simplify the notation. 
	Assume moreover there exists an $\wreathOfSym{k-1}{{\bf n}'}$-equivariant parametrisation map
	\begin{equation} \label{eq: p_{k-1}}
		p_{k-1}: \wreathOfConf{k-1}{d}{{\bf n}'} \times \K^d \longrightarrow (\K^d)^{\times n_1 \cdots n_{k-1}}
	\end{equation}
	with the required slice-wise restrictions. The map $p_k$ is constructed from two ingredients. 
	\begin{compactenum}[\quad \rm (1)]
		\item Let the $\wreathOfSym{k}{\bf n}$-equivariant map
	\begin{equation} \label{eq: first map in the partition map}
		\id \times p_{1}:\wreathOfConf{k-1}{d}{{\bf n}'}^{\times n_k} \times \conf(\R^d, n_k) \times \K^d \longrightarrow \wreathOfConf{k-1}{d}{{\bf n}'}^{\times n_k} \times (\K^d)^{\times n_k}
	\end{equation}
	be the product of the identity on $\wreathOfConf{k-1}{d}{{\bf n}'}^{\times n_k}$ and the map \eqref{eq: partition map with sites} with value $n = n_k$. The action of the group $\wreathOfSym{k}{{\bf n}} = \wreathOfSym{k-1}{{\bf n}'}^{\times n_k} \rtimes \sym_{n_k}$ on the codomain of the map \eqref{eq: first map in the partition map} is induced by the product action of $\wreathOfSym{k-1}{{\bf n}'}^{\times n_k}$ on $\wreathOfConf{k-1}{d}{{\bf n}'}^{\times n_k}$, and the action of $\sym_{n_k}$ on the product $(\K^d)^{\times n_k}$ which permutes the coordinates. Notice that the map \eqref{eq: first map in the partition map} restricts to the slice-wise equivariant map
	\[
		\wreathOfConf{k}{d}{{\bf n}} \times \{K\} \longrightarrow \wreathOfConf{k-1}{d}{{\bf n}'}^{\times n_k} \times \EMP_{\mu}(K, n_k) \subseteq \wreathOfConf{k-1}{d}{{\bf n}'}^{\times n_k} \times (\K^d)^{\times n_k}
	\]
	for each $K \in \K^d$.
		\item The $n_k$-fold product $(p_{k-1})^{\times n_k}$ of the map \eqref{eq: p_{k-1}} induces the $\wreathOfSym{k}{{\bf n}}$-equivariant map
		\begin{equation} \label{eq: second map in the partition map}
			\wreathOfConf{k-1}{d}{{\bf n}'}^{\times n_k} \times (\K^d)^{\times n_k} \xrightarrow{~(p_{k-1})^{\times n_k}~} ((\K^d)^{\times n_1 \cdots n_{k-1}})^{\times n_k}.
		\end{equation}
		From the induction hypothesis it follows that the map \eqref{eq: second map in the partition map} restricts to the slise-wise equivariant map
		\[
			\wreathOfConf{k-1}{d}{{\bf n}'}^{\times n_k} \times \{(K_1, \dots ,K_{n_k})\} \longrightarrow \prod_{i = 1}^{n_k} \EMP_{\mu}(K_i, n_1\cdots n_{k-1}) \subseteq \prod_{i = 1}^{n_k} (\K^d)^{\times n_1 \cdots n_{k-1}}
		\]
		for each $(K_1, \dots ,K_{n_k}) \in (\K^d)^{\times n_k}$.
	\end{compactenum}
	Having these two ingredients, we can define the map \eqref{eq: p_k} as the $\wreathOfSym{k}{n_1, \dots ,n_k}$-equivariant composition
	\begin{equation*}
			p_k: \wreathOfConf{k}{d}{\bf n} \times \K^d \xrightarrow{~\eqref{eq: first map in the partition map}~} \wreathOfConf{k-1}{d}{{\bf n}'}^{\times n_k} \times (\K^d)^{\times n_k} \xrightarrow{~\eqref{eq: second map in the partition map}~} (\K^d)^{\times n_1 \cdots n_{k}}.
	\end{equation*}
	The slice-wise restrictions of the maps \eqref{eq: first map in the partition map} and \eqref{eq: second map in the partition map} imply that the restriction of the map $p_k$ factors as
	\[
		\wreathOfConf{k}{d}{\bf n} \times \{K\} \longrightarrow \EMP_{\mu}(K, n_1 \cdots n_k)
	\]
	for each $K \in \K^d$, which finishes the proof.
\end{proof}

\medskip
In order to proceed with setting up the configuration space -- test map scheme, we first define an $\wreathOfSym{k}{n_1, \dots ,n_k}$-representation which is used as a codomain of the test map of the scheme.

\medskip
For each $n \ge 1$ let $W_n := \{ x \in \R^n:~ x_1 + \dots + x_n = 0 \}$. It might be convenient to consider $W_n$ to be a space of row vectors with zero coordinate sums.

\medskip
\begin{definition}
Let $d \ge 1$, $k\ge 1$ and $n_1, \dots ,n_k \ge 2$ be integers.
The vector space $\wreathOfW{k}{d-1}{n_1, \dots ,n_k}$ is defined inductively as follows.
	\begin{compactitem}[\quad --]
		\item For $k=1$ set 
		\[
		\wreathOfW{1}{d-1}{n_1} := W_{n_1}^{\oplus d-1} \cong \{(z_1, \dots ,z_{n_1}) \in (\R^{d-1})^{\oplus n_1}:~ z_1 + \dots + z_{n_1} = 0\}.
		\]
		
		\item For $k \ge 2$ let
		\[
			\wreathOfW{k}{d-1}{n_1, \dots ,n_k} := \wreathOfW{k-1}{d-1}{n_1, \dots ,n_{k-1}}^{\oplus n_k} \oplus W_{n_k}^{\oplus d-1}.
		\]
	\end{compactitem}
\end{definition}

\medskip
The dimension of the vector space in the case $k=1$ is $\dim \wreathOfW{1}{d-1}{n_1} = (d-1)(n_1-1)$. 
From the inductive relation
\[
	\dim \big(\wreathOfW{k}{d-1}{n_1, \dots ,n_k}\big) = n_k \cdot \dim \big( \wreathOfW{k-1}{d-1}{n_1, \dots ,n_{k-1}}\big) + (d-1)(n_k-1),
\]
where $k \ge 2$, it follows that
\[
	\dim \big(\wreathOfW{k}{d-1}{n_1, \dots ,n_k}\big) = (d-1)(n_1 \cdots n_k - 1).
\] 
In the analogy to the case of $\wreathOfConf{k}{d}{n_1, \dots , n_k}$, the action of $\wreathOfSym{k}{n_1, \dots, n_k}$ on $\wreathOfW{k}{d-1}{n_1, \dots ,n_k}$ is defined by the wreath product action; consult Definition \ref{def: wreath product action}.

\medskip
Let $S(\wreathOfW{k}{d-1}{n_1, \dots ,n_k})$ be the unit sphere of $\wreathOfW{k}{d-1}{n_1, \dots ,n_k}$.
Now, we state the main theorem of this section which establishes the configuration space -- test map scheme.

\medskip
\begin{theorem} \label{theorem: cs-tm scheme}
	Let $k \ge 1$, $d\ge 1$ and $n_1, \dots ,n_k \ge 2$ be integers. 
	Let $K \in \K^d$ be a $d$-dimensional convex body and $n := n_1 \cdots n_k$. 
	If there is no $\wreathOfSym{k}{n_1, \dots ,n_k}$-equivariant map of the form
	\begin{equation} \label{eq: wreath to sphere of wreath map}
		\wreathOfConf{k}{d}{n_1, \dots ,n_k} \longrightarrow S(\wreathOfW{k}{d-1}{n_1, \dots ,n_k}),
	\end{equation}
	then there exists a solution to Conjecture \ref{conjecture: nandakumar and ramana rao for general d} for the convex body $K$ which is of level $k$ and type $(n_1, \dots ,n_k)$.
\end{theorem}
\begin{proof}
	Let us first construct an $\wreathOfSym{k}{n_1, \dots ,n_k}$-equivariant map
	\begin{equation} \label{eq: EMP to wreathOfW}
		\Phi_k:\EMP_{\mu}(K ,n_1 \cdots n_k) \longrightarrow \wreathOfW{k}{d-1}{n_1, \dots ,n_k},
	\end{equation}
	which tests whether a convex equipartition of $K$ into $n_1 \cdots n_k$ parts is a solution to Conjecture \ref{conjecture: nandakumar and ramana rao for general d}. 
	The map $\Phi_k$ will be induced from the functions $\varphi_1, \dots ,\varphi_{d-1}: \K^d \rightarrow \R$, which are given in the statement of the conjecture, and will be defined inductively on $k \ge 1$. 
	In fact, it will be given on a larger domain $(\K^d)^{\times n_1 \cdots n_k}$, where the action of the group $\wreathOfSym{k}{n_1,...,n_k}$ is as before.
	
	\medskip
	For $k=1$, the $\sym_{n_1}$-equivariant map
     \[
        \Phi_1: (\K^d)^{\times n_1} \longrightarrow W_{n_1}^{\oplus d-1}
     \]
     is given by mapping $(K_1,...,K_{n_1})$ to the an element which, for each $1 \le i \le d-1$, has in the coordinate living in the $i$th copy of $W_{n_1}$ the vector obtained from 
     \[
        (\varphi_i(K_1), \dots , \varphi_i(K_{n_1})) \in \R^{n_1}
    \]
    by subtracting the average $\frac{1}{n}(\varphi_i(K_1)+ \dots + \varphi_i(K_{n_1}))$ from each coordinate.

	\medskip
    Suppose $k \ge 2$ and assume the $\wreathOfSym{k-1}{n_1, \dots , n_{k-1}}$-equivariant map
    \[
        \Phi_{k-1}: (\K^d)^{\times n_1 \cdots n_{k-1}} \longrightarrow \wreathOfW{k-1}{d-1}{n_1, \dots , n_{k-1}}
    \]
    is already constructed . 
    We define the desired $\wreathOfSym{k}{n_1, \dots ,n_k}$-equivariant map
    \[
        \Phi_k: ((\K^d)^{\times n_1 \cdots n_{k-1}})^{\times n_k} \longrightarrow \wreathOfW{k-1}{d-1}{n_1, \dots ,n_{k-1}}^{\oplus n_k} \oplus W_{n_k}^{\oplus d-1}
    \]
    by mapping an element $(\Pi_1,...,\Pi_{n_k}) \in ((\K^d)^{\times n_1 \cdots n_{k-1}})^{\times n_k}$ to an element which:
    \begin{itemize}[\quad --]
        \item for each $1 \le j \le n_k$, at the coordinate living in the $j$th copy of $\wreathOfW{k-1}{d-1}{n_1, \dots ,n_{k-1}}$ has the value $\Phi_{k-1}(\Pi_j)$, and
        \item for each $1 \le i \le d-1$, at the coordinate living in the $i$th copy of $W_{n_k}$ is the vector obtained from
        \[
            (\varphi^{\Sigma}_i(\Pi_1), \dots , \varphi^{\Sigma}_i(\Pi_{n_k})) \in \R^{n_k}
        \]
        by subtracting the average $\frac{1}{n_k}(\varphi^{\Sigma}_i(\Pi_1)+ \dots + \varphi^{\Sigma}_i(\Pi_{n_k}))$ from each coordinate.
        Here we denoted by $\varphi_i^{\Sigma}$ the composition
        \[
            \varphi_i^{\Sigma}: (\K^d)^{\times m} \xrightarrow{~~ (\varphi_i)^{\times m} ~} \R^m \xrightarrow{\langle -,~ (1, \dots , 1) \rangle} \R,
        \]
        for any integer $m \ge 1$.  In the current situation $m=n_1\cdots n_{k-1}$.  
        \end{itemize}
        This completes the construction of the map $\Phi_k$.
        From the induction hypothesis it follows that $\Phi_k$ is indeed $\wreathOfSym{k}{n_1, \dots , n_k}$-equivariant.
    
\medskip    
The crucial property of the map \eqref{eq: EMP to wreathOfW} is that it maps a tuple $(K_1, \dots, K_n) \in (\K^d)^{\times n}$ to the origin if and only if the tuple satisfies
\[
    \varphi_i(K_1) = \dots = \varphi_i(K_n)
\]
for each $1 \le i \le n$.

\medskip
Now, we return to the proof of the theorem. 
Assume that for a convex body $K \in \K^d$ there is no solution to Conjecture \ref{conjecture: nandakumar and ramana rao for general d} which is of level $k$ and type $(n_1, \dots ,n_k)$. 
We will show that there exists an $\wreathOfSym{k}{n_1, \dots ,n_k}$-equivariant map of the form \eqref{eq: wreath to sphere of wreath map}.
 
Namely, since $K$ does not solve the conjecture, the $\wreathOfSym{k}{n_1, \dots , n_k}$-equivariant map \eqref{eq: EMP to wreathOfW} does not hit the origin.
In other words we may restrict the codomain of the map $ \Phi_k$
\[
    \Phi_k\colon \EMP_{\mu}(K, n) \longrightarrow \wreathOfW{k}{d-1}{n_1, \dots , n_k} \setminus \{ 0 \}.
\]
(In order to simplify the notation we do not introduce new name for the new map.)
Pre-composing it with the $\wreathOfSym{k}{n_1, \dots ,n_k}$-equivariant parametrisation map
\[
	\wreathOfConf{k}{d}{n_1, \dots ,n_k} \longrightarrow \EMP_{\mu}(K,n)
\]
from Proposition \ref{proposition: parametrisation of iterated partitions}, and post-composing it with the $\wreathOfSym{k}{n_1, \dots ,n_k}$-equivariant retraction
	\[
		\wreathOfW{k}{d-1}{n_1, \dots ,n_k} \setminus \{0\} \longrightarrow S(\wreathOfW{k}{d-1}{n_1, \dots ,n_k}),
	\]
yields an $\wreathOfSym{k}{n_1, \dots ,n_k}$-equivariant map
\[
    \wreathOfConf{k}{d}{n_1, \dots ,n_k} \longrightarrow \EMP_{\mu}(K,n) \longrightarrow \wreathOfW{k}{d-1}{n_1, \dots , n_k} \setminus \{ 0 \} \longrightarrow S(\wreathOfW{k}{d-1}{n_1, \dots ,n_k})
\]
of the form \eqref{eq: wreath to sphere of wreath map}, which completes  the proof.
\end{proof}

\medskip
\section{Proof of Theorem \ref{theorem: main}: Equivariant Obstruction Theory} 
\label{section: equivariant obstruction theory}
\medskip

In this section we give a proof of Theorem \ref{theorem: main} using equivariant obstruction theory of tom Dieck \cite[Sec.\,II.3]{tomDieck:TransformationGroups}. 
Namely, for integers $k \ge 1$, $d \ge 1$ and $n_1, \dots , n_k \ge 2$ we show that an $\wreathOfSym{k}{n_1, \dots ,n_k}$-equivariant map
\begin{equation*}
	\wreathOfConf{k}{d}{n_1, \dots , n_k} \longrightarrow S(\wreathOfW{k}{d-1}{n_1, \dots , n_k})
\end{equation*}
does not exist if and only if $n_1, \dots , n_k$ are powers of the same prime number.
 
In order to use the equivariant obstruction theory \cite[Sec.\,II.3]{tomDieck:TransformationGroups}, we need to  construct a cellular model $\cellWreathOfConf{k}{d}{n_1, \dots , n_k}$ of the space $\wreathOfConf{k}{d}{n_1, \dots , n_k}$ which is also its equivariant deformation retract. 
Then, instead of studying the existence of the above map, we focus our attention on the equivalent problem of the existence of an $\wreathOfSym{k}{n_1, \dots , n_k}$-equivariant map
\begin{equation*}
	\cellWreathOfConf{k}{d}{n_1, \dots , n_k} \longrightarrow S(\wreathOfW{k}{d-1}{n_1, \dots , n_k}).
\end{equation*}

\medskip
\subsection{Cellular model}

Blagojevi\'c \& Ziegler \cite[Sec.\,3]{BlagojevicZiegler15} have constructed an $(n-1)(d-1)$-dimensional $\sym_n$-equivariant cellular model $\cellConf(d,n)$ of the Salvetti type for the classical configuration space $\conf(\R^d, n)$ which is its  $\sym_n$-equivariant deformation retract. 
In this section, we use this model to construct an $\wreathOfSym{k}{n_1, \dots ,n_k}$-equivariant model $\cellWreathOfConf{k}{d}{n_1, \dots , n_k}$ of $\wreathOfConf{k}{d}{n_1, \dots , n_k}$ which is of dimension $(d-1)(n_1 \dots  n_k-1)$ and is its $\wreathOfSym{k}{n_1, \dots , n_k}$-equivariant deformation retract.
Additionally we collect relevant facts about this model needed for an application of the equivariant obstruction theory.

\medskip
For integers $k \ge 1$, $d \ge 1$ and $n_1, \dots , n_k \ge 2$ we simplify notation on occasions by setting
${\bf n} := (n_1, \dots ,n_k)$ and ${\bf n}' := (n_1, \dots, n_{k-1})$.
However, we will keep the longer notation in statements of lemmas, definitions, propositions, and theorems.

\medskip
\begin{definition} [Cellular model]
	Let $k \ge 1$, $d \ge 1$ and $n_1, \dots , n_k \ge 2$ be integers.
	An $\wreathOfSym{k}{n_1, \dots , n_k}$-equivariant cell complex $\cellWreathOfConf{k}{d}{n_1, \dots , n_k}$ is defined inductively as follows.
	 \begin{compactitem}[\quad --]
		\item For $k = 1$ set $\cellWreathOfConf{1}{d}{n_1} := \cellConf(\R^d, n_1)$ with the $\wreathOfSym{1}{n_1} \cong \sym_{n_1}$-action (see \cite[Sec.\,3]{BlagojevicZiegler15}).
		
		\item For $k \ge 2$ set
		\[
			\cellWreathOfConf{k}{d}{n_1, \dots , n_k} := \cellWreathOfConf{k-1}{d}{n_1, \dots , n_{k-1}}^{\times n_k} \times \cellConf(d, n_k)
		\]
		with the wreath product action of $\wreathOfSym{k}{n_1, \dots , n_k} = \wreathOfSym{k-1}{n_1, \dots , n_{k-1}}^{\times n_k} \rtimes \sym_{n_k}$ (see \mbox{Definition \ref{def: wreath product action}}).
	\end{compactitem}
\end{definition}

\medskip
From the recursive formula follows that 
\[
	\dim \cellWreathOfConf{k}{d}{\bf n} = n_k \dim \cellWreathOfConf{k-1}{d}{{\bf n}'} + (d-1)(n_k-1).
\]
Since in the base case $\dim \cellWreathOfConf{1}{d}{n} = \dim \cellConf(d, n) = (d-1)(n-1)$, it follows that the dimension of the complex $\cellWreathOfConf{k}{d}{\bf n}$ is indeed 
\[
	M_k := (d-1)(n_1 \dots n_k-1).
\]
Even though $M_k$ depends also on parameters $d$ and $n_1, \dots , n_k$, for the sake of simplicity
 they will be omitted from the notation. However, we will keep parameter $k$ in the notation, as most of the proofs are done using induction on $k$. This convention will be kept for other notions when there is not danger of confusion.

\medskip
\begin{proposition} 
	Let $k \ge 1$, $d \ge 1$ and $n_1, \dots , n_k \ge 2$ be integers. 
	Then $\cellWreathOfConf{k}{d}{n_1, \dots , n_k}$ is a finite, regular, free $\wreathOfSym{k}{n_1, \dots ,n_k}$-CW-complex. 
	Moreover, it is an $\wreathOfSym{k}{n_1, \dots , n_k}$-equivariant deformation retract of the wreath product of configuration spaces $\wreathOfConf{k}{d}{n_1, \dots , n_k}$.
\end{proposition}
\begin{proof}
	The proof is done by induction on $k \ge 1$. For the base case $k = 1$ let 
	\[
		r_1\colon \conf(\R^d, n) \longrightarrow_{\sym_n} \cellConf(\R^d, n)
	\]
	be the $\wreathOfSym{1}{n_1}$-equivariant deformation retraction and $h_1\colon \id \simeq_{\sym_n} i_1 \circ r_1$ the equivariant homotopy guarantied by \cite[Thm.\,3.13]{BlagojevicZiegler15}.
	Here  $i_1\colon \cellConf(d,n) \hookrightarrow \conf(\R^d, n)$ is the inclusion.
	
	\medskip
	Let now $k \ge 2$. 
	The deformation retraction $r_k$ is set to be
	\[
		r_k := (r_{k-1})^{\times n_k} \times r_1\colon  \wreathOfConf{k-1}{d}{{\bf n}'}^{\times n_k} \times \conf(\R^d, n_k) \to \cellWreathOfConf{k-1}{d}{{\bf n}'}^{\times n_k} \times \cellConf(d, n_k).
	\] 
	By the induction hypothesis it follows that the complex is finite, regular and free, as well as the fact that $r_k$ is indeed an $\wreathOfSym{k}{\bf n}$-equivariant retraction.
	Furthermore, the homotopy defined by
	\[
		h_k := (h_{k-1})^{\times n_k} \times h_1\colon \id \simeq_{\wreathOfSym{k}{\bf n}} i_k \circ r_k
	\]
	makes $r_k$ into a deformation retraction.
	Here, $i_k\colon  \cellWreathOfConf{l}{d}{\bf n} \hookrightarrow \wreathOfConf{k}{d}{\bf n}$ is the inclusion. 
\end{proof}

\medskip
Let us describe in more detail the cellular structure of the $\sym_n$-CW complex $\cellConf(d,n)$ developed by Blagojevi\'c \& Ziegler in \cite[Thm.\,3.13]{BlagojevicZiegler15}. 
The set of cells of the $\sym_n$-CW complex complex $\cellConf(d, n)$ is 	
\[
	\{\check{c}(\sigma, {\bf i}):~ \sigma \in \sym_n, ~ {\bf i} \in [d]^{\times n-1}\}.
\]
The dimension of the cell $\check{c}(\sigma, {\bf i})$ is given by the formula  $\dim \check{c}(\sigma, {\bf i}) = (i_1-1) + \dots + (i_{n-1}-1)$, making the complex $(n-1)(d-1)$-dimensional. 
The cellular $\sym_n$-action is given on the cells by $\tau \cdot \check{c}(\sigma, {\bf i}) = \check{c}(\tau \cdot \sigma, {\bf i})$, for each $\tau \in \sym_n$. 
In particular, from this information, we obtain the following:
\begin{compactitem}[\quad --]
	\item There are in total $n!$ maximal cells, all of which belong to the same $\sym_n$-orbit.  
	Let us choose an orbit representative $\check{c}(\id, {\bf d})$, where ${\bf d} := (d, \dots ,d) \in [d]^{\times n-1}$. 
	\item There are in total $(n-1)n!$ codimension one cells split into $n-1$ orbits of the group $\sym_n$. 
	For $1 \le i \le n-1$ let us denote by ${\bf d}_i \in [d]^{\times n-1}$ the vector which has value $d$ on all coordinates different from $i$, and value $d-1$ on coordinate $i$. In this notation, codimension one cells are encoded by the set
		\[
			\{\check{c}(\sigma, {\bf d}_i):~ \sigma \in \sym_n,~ 1 \le i \le n-1\}.
		\]
		For example, one choice of $n-1$ orbit representatives is $\check{c}(\id, {\bf d}_1), \dots , \check{c}(\id, {\bf d}_{n-1})$.
\end{compactitem}

\medskip
Boundary of a maximal cell $\check{c}(\sigma, {\bf d})$ consists of the following $\binom{n}{1} + \dots + \binom{n}{n-1}$ codimension one cells, as can be seen either from the boundary relations in \cite[Thm.\,3.13]{BlagojevicZiegler15} or explicitly from the proof of \mbox{\cite[Lem.\,4.1]{BlagojevicZiegler15}}. 
For $\sigma \in \sym_n$, integer $1 \le m \le n-1$, and an $m$-element subset $J \subseteq [n]$, let us denote by $\sigma_J \in \sym_n$ the permutation
\[
	\sigma_J\colon [n] \longrightarrow [n],\qquad t \longmapsto \sigma(j_t),
\]
where $J = \{j_1 < \dots < j_m)$ and $[n] \setminus J = \{j_{m+1}< \dots < j_n\}$ is the ordering of the elements. The set of codimension one cells in the boundary of the maximal cell $\check{c}(\sigma, {\bf d})$ is encoded by the set 
\begin{equation} \label{eq: bdry cell of ch(id, d)}
	\{\check{c}(\sigma_J, {\bf d}_{m}):~ 1 \le m \le n-1,~ J  \subseteq [n] \text{ with } |J| = m\},
\end{equation}
where two cells $\check{c}(\sigma_J, {\bf d}_{m})$ and $\check{c}(\sigma_I, {\bf d}_{s})$ belong to the same $\sym_n$-orbit if and only if $m = s$. In particular, for each $1 \le m \le n-1$, the size of the intersection of the orbit of the cell $\check{c}(\id, {\bf d}_m)$ with the above set of boundary cells is precisely $\binom{n}{m}$.

\medskip
Extending the notation introduced above, let us define a generalisation of the special cell $\check{c}(\id, {\bf d})$ of $\cellConf(d, n)$, which was chosen to be the representative of the orbit of maximal cells.

\medskip
\begin{definition}[Orbit representative maximal cell] \label{def: orbit representative maximal cell}
	Let $k \ge 1$, $d \ge 1$ and $n_1, \dots , n_k \ge 2$ be integers. 
	We define a maximal cell $e_k$ of the cell complex $\cellWreathOfConf{k}{d}{n_1, \dots , n_k}$ inductively as follows.
	\begin{compactitem}[\quad --]
		\item For $k = 1$ let $e_1 := \check{c}(\id, {\bf d})$ be a maximal cell in $\cellWreathOfConf{1}{d}{n} = \cellConf(d,n)$.
		
		\item For $k \ge 2$ let us define 
		\[
			e_k := (e_{k-1})^{\times n_k} \times \check{c}(\id, {\bf d})
		\]
		to be a maximal cell in cell complex 
		\[
			\cellWreathOfConf{k}{d}{n_1, \dots , n_k} = \cellWreathOfConf{k-1}{d}{n_1, \dots , n_{k-1}}^{\times n_k} \times \cellConf(d, n_k),
		\]
		where $e_{k-1}$ denotes the previously defined maximal cell of $\cellWreathOfConf{k-1}{d}{n_1, \dots , n_{k-1}}$ and $\check{c}(\id, {\bf d})$ is a maximal cell of $\cellConf(d, n_k)$.
	\end{compactitem}
\end{definition}

\medskip
In the next lemma we describe the index set for the codimension one cells lying in the boundary of the $M_k$-cell $e_k$.

\medskip
\begin{lemma} \label{lem: orbits of M and (M-1)-cells and B_k}
	Let $k \ge 1$, $d \ge 1$ and $n_1, \dots , n_k \ge 2$ be integers.
	\begin{enumerate}[\quad \rm (i)]
		\item\label{lem: orbits of M and (M-1)-cells and B_k -- 01} All maximal dimensional cells of $\cellWreathOfConf{k}{d}{n_1, \dots , n_k}$ form a single $\wreathOfSym{k}{n_1, \dots , n_k}$-orbit.
		
		\item\label{lem: orbits of M and (M-1)-cells and B_k -- 02} The orbit representative maximal cell $e_k$ chosen in Definition \ref{def: orbit representative maximal cell} has the boundary consisting of codimension one cells which are indexed by the set
			\begin{equation*}
				B_k := \bigcup_{i = 1}^k \bigcup_{m = 1}^{n_i - 1} \binom{[n_i]}{m} \times [n_{i+1}] \times \dots \times [n_{k}].
			\end{equation*}
		
		\item\label{lem: orbits of M and (M-1)-cells and B_k -- 03} The $\wreathOfSym{k}{n_1, \dots , n_k}$-orbit stratification of $B_k$ is given by
			\[
				\Big\{\binom{[n_i]}{m} \times \{{\bf j}\}:~ 1 \le i \le k,~ 1 \le m \le n_i - 1,~ {\bf j} \in [n_{i+1}] \times \dots \times [n_{k}] \Big\}.
			\]
	\end{enumerate}
\end{lemma}
\begin{proof}
	The proof of all three statements is done simultaneously by induction on $k \ge 1$. 
	The base case $k = 1$ of the complex $\cellWreathOfConf{1}{d}{n} = \cellConf(d, n)$ is treated in \cite[Lem.\,4.1]{BlagojevicZiegler15} and \cite[Proof of Lem.\,4.2]{BlagojevicZiegler15}. 
	See also the boundary description \eqref{eq: bdry cell of ch(id, d)}. 
	A codimension one boundary cell $\check{c}(\id_J, {\bf d}_j) \subseteq \partial \check{c}(\id, {\bf d})$ corresponds to an element
	\[
		J \in \bigcup_{m=1}^{n-1} \binom{[n]}{m} = B_1,
	\]
	for each $1 \le j \le n-1$ and $J \subseteq [n]$ with $|J| = j$. 
	
	\medskip
	Let $k \ge 2$. 
	From the inductive definition of the cell complex $\cellWreathOfConf{k}{d}{\bf n}$ and the induction hypothesis it follows that the $M_k$-cells form a single $\wreathOfSym{k}{\bf n}$-orbit, which completes the proof of part \eqref{lem: orbits of M and (M-1)-cells and B_k -- 01}. 
	By the boundary formula of the product applied to the cell $e_k = (e_{k-1})^{n_k} \times \check{c}(\id, {\bf d})$ it follows that the $(M_k-1)$-cells in the boundary $\partial e_k$ are of the following two types.
	\begin{compactenum} [\quad (1)]
		\item The first type is
		\[
			(e_{k-1})^{ i-1} \times \beta \times (e_{k-1})^{ n_k-i} \times \check{c}(\id, {\bf d}),
		\]
		where $1 \le i \le n_k$ and $\beta$ is the codimension one boundary cell of $e_{k-1}$. Let $\beta \subseteq \partial e_{k-1}$ be indexed by $b \in B_{k-1}$. The above $(M_k-1)$-cell is set to be indexed by 
		\[
			(b, i) \in B_{k-1} \times [n_k].
		\]
		From the inductive definition of the $\wreathOfSym{k}{\bf n}$-action it follows that two boundary cells of the first type indexed by 
		\[
			(b, i),~ (b', i') \in B_{k-1} \times [n_k]
		\]
		are in the same $\wreathOfSym{k}{\bf n}$-orbit if and only if $i = i'$, and cells $\beta,\beta' \subseteq \partial e_{k-1}$ are in the same $\wreathOfSym{k-1}{{\bf n}'}$-orbit.
		
		\item The second type is
		\[
			(e_{k-1})^{ n_k} \times \check{c}(\id_J, {\bf d}_j)
		\]
		where $1 \le j \le n_k - 1$ and $J \subseteq [n_k]$ with $|J| = j$. Here, $\check{c}(\id_J, {\bf d}_j)$ denotes the codimension one boundary cell of $\check{c}(\id, {\bf d})$ described in \eqref{eq: bdry cell of ch(id, d)}. The above $(M_k-1)$-cell is set to be indexed by 	
		\[
			J \in \bigcup_{m=1}^{n_k-1} \binom{[n_k]}{m}.
		\] 
		From the definition of the group action it follows that two boundary cells of the second type, indexed by non-empty proper subsets $J, J' \subseteq [n_k]$, are in the same $\wreathOfSym{k}{\bf n}$-orbit if and only if $|J| = |J'|$.
	\end{compactenum}
	In particular, we obtained a recursive formula
	\[
		B_k = B_{k-1} \times [n_k] ~\cup~  \bigcup_{m=1}^{n_k-1} \binom{[n_k]}{m},
	\]
	which together with the base case description of $B_1$ implies the part \eqref{lem: orbits of M and (M-1)-cells and B_k -- 02}.
	
	\medskip
	Since no cell of the first type is in the same $\wreathOfSym{k}{\bf n}$-orbit as the cells of the second type, by the above orbit description of the cells of each of the two types, it follows that the orbit stratification of $B_k$ stated in the part \eqref{lem: orbits of M and (M-1)-cells and B_k -- 03} holds.
\end{proof}

\medskip
\subsection{Obstructions}

Our goal is to show that an $\wreathOfSym{k}{\bf n}$-equivariant map
\begin{equation} \label{eq: cellWreath to sphere}
	\cellWreathOfConf{k}{d}{\bf n} \longrightarrow S(\wreathOfW{k}{d-1}{\bf n})
\end{equation}
does not exists if and only if $n_1, \dots , n_k$ are all powers of the same prime number. 
We have that:
\begin{compactitem}[\quad --]
	\item $\cellWreathOfConf{k}{d}{\bf n}$ is a free $\wreathOfSym{k}{\bf n}$-CW complex of dimension $M_k$,
	
	\item the sphere $S(\wreathOfW{k}{d-1}{\bf n})$ is $(M_k-1)$-simple and $(M_k-2)$-connected.
\end{compactitem}
Applying the equivariant obstruction theory \cite[Sec.\,II.3]{tomDieck:TransformationGroups}, we get that the existence of an $\wreathOfSym{k}{\bf n}$-equivariant map \eqref{eq: cellWreath to sphere} is equivalent to the vanishing of the primary obstruction class
\[
	{\frak o} = [c(f_k)] \in H^{M_k}_{\wreathOfSym{k}{\bf n}}(\cellWreathOfConf{k}{d}{\bf n}; \pi_{M_k-1}(S(\wreathOfW{k}{d-1}{\bf n}))).
\]
Here $c(f_k)$ represents the $M_k$-dimensional obstruction cocycle associated to a general position equivariant map \mbox{$f_k\colon\cellWreathOfConf{k}{d}{\bf n} \longrightarrow \wreathOfW{k}{d-1}{\bf n}$} (see \cite[Def.\,1.5]{BlagojevicBlagojevic07}). 
Its values on the $M_k$-cells $e$ are given by the following degrees
\[
	c(f_k)(e) = \deg(r \circ f_k\colon \partial e \longrightarrow \wreathOfW{k}{d-1}{\bf n} \setminus \{0\} \longrightarrow S(\wreathOfW{k}{d-1}{\bf n})),
\]
where $r$ is the radial retraction.

\medskip
The Hurewicz isomorphism \cite[Cor.\,VII.10.8]{Bredon2010} gives an $\wreathOfSym{k}{\bf n}$-module isomorphism
\[
	\pi_{M_k-1}(S(\wreathOfW{k}{d-1}{\bf n})) \cong H_{M_k-1}(S(\wreathOfW{k}{d-1}{\bf n});\Z)  =:\ZZ_k(d-1; {\bf n}).
\]
In order to simplify the notation, let us put $\ZZ_k := \ZZ_k(d-1; {\bf n})$ when there is no danger of confusion.

The $\wreathOfSym{k}{\bf n}$-module $\ZZ_k = \langle \xi_k \rangle$ is isomorphic to $\Z$ as an abelian group. 
In order to describe the $\wreathOfSym{k}{\bf n}$-module structure on $\ZZ_k$ we need the following.

\medskip
\begin{definition}(Orientation function) \label{def: orientation function}
	Let $k \ge 1$, $d \ge 1$ and $n_1, \dots , n_k \ge 2$ be integers. 
	The orientation function
	\[
		\orient\colon \wreathOfSym{k}{n_1, \dots , n_k} \longrightarrow \{-1, +1\}
	\]
	is given inductively on $k \ge 1$ as follows.
	\begin{compactitem}[\quad --]
		\item For $k = 1$ we set $\orient(\sigma) := (\sgn \sigma)^{d-1}$ for each $\sigma \in \sym_n = \wreathOfSym{1}{n}$, where $\sgn$ denotes the sign of the permutation.
		
		\item For $k \ge 2$ we set
		\[
			\orient(\Sigma_1, \dots, \Sigma_{n_k}; \sigma) := (\sgn \sigma)^{(d-1)\cdot n_1 \cdots n_{k-1}} \cdot \orient(\Sigma_1) \cdots \orient(\Sigma_{n_k}),
		\]
		for any 
		\[
			(\Sigma_1, \dots, \Sigma_{n_k}; \sigma) \in \wreathOfSym{k-1}{n_1, \dots , n_{k-1}}^{\times n_k} \rtimes \sym_{n_k} = \wreathOfSym{k}{n_1, \dots , n_k}.
		\]
	\end{compactitem}
\end{definition}

\medskip
In the next lemma we show that the action of $\wreathOfSym{k}{\bf n}$ changes orientation on the vector space $\wreathOfW{k}{d-1}{\bf n}$ according to the orientation function $\orient$. 
Consequently, the $\wreathOfSym{k}{\bf n}$-module structure on $\ZZ_k = \langle \xi_k \rangle$ is given by
\begin{equation} \label{eq: wreathSym-module structure on ZZ}
	\Sigma \cdot \xi_k = \orient(\Sigma) \cdot \xi_k
\end{equation}
for each $\Sigma \in \wreathOfSym{k}{\bf n}$.

\medskip
\begin{lemma} \label{lem: wreathSym action on wreathW}
	Let $k \ge 1$, $d \ge 1$ and $n_1, \dots , n_k \ge 2$ be integers. 
	The group $\wreathOfSym{k}{\bf n}$ acts on the vector space $\wreathOfW{k}{d-1}{\bf n}$ by changing the orientation according to the orientation function $\orient$ from Definition \ref{def: orientation function}. In particular, the map
	\[
		\orient: \wreathOfSym{k}{n_1, \dots , n_k} \longrightarrow (\{-1, +1\}, \cdot)
	\]
	is a group homomorphism.
\end{lemma}
\begin{proof}
	The proof is conducted by induction on $k \ge 1$. 
	
	\medskip
	The base case $k = 1$ is treated in \cite[Sec.\,4, p.\,69]{BlagojevicZiegler15}. 
	Indeed, each transposition $\tau_{ij} \in \sym_n$ acts on $W_n$ by reflection in the hyperplane $x_i = x_j$, so a permutation $\sigma \in \sym_n$ reverses the orientation on $W_n$ by $\sgn \sigma$. 
	Thus, $\sigma$ changes the orientation of $W_n^{\oplus d-1}$ by $(\sgn \sigma)^{d-1}$.
	
	\medskip
	Let $k \ge 2$. 
	We can split the action of an element 
	\[
		((\Sigma_i); \sigma) := (\Sigma_1, \dots , \Sigma_{n_k};\sigma) \in \wreathOfSym{k-1}{{\bf n}'}^{\times n_k} \rtimes \sym_{n_k} = \wreathOfSym{k}{\bf n}
	\]
	on a vector 
	\[
		((V_i); v) := (V_1, \dots , V_{n_k};v) \in \wreathOfW{k-1}{d-1}{{\bf n}'}^{\oplus n_k} \oplus W_{n_k}^{\oplus d-1} = \wreathOfW{k}{d-1}{\bf n}
	\]
	in two steps $(\star)$ and $(\star\star)$ as follows. 
	We have
	\[
		((\Sigma_i);\sigma) \cdot ((V_i); v) = ((\Sigma_i); \id) \cdot ((\id); \sigma) \cdot ((V_i); v) \overset{(\star)}{=} ((\Sigma_i); \id) \cdot ((V_{\sigma^{-1}(i)}); \sigma v) \overset{(\star\star)}{=} ((\Sigma_iV_{\sigma^{-1}(i)}); \sigma v).
	\]
	In the step $(\star)$ permutation $\sigma \in \sym_{n_k}$ acts on $v \in W_{n_k}^{\oplus d-1}$ by changing the orientation by $(\sgn \sigma)^{d-1}$, and permutes the $n_k$ vectors $V_i \in \wreathOfW{k-1}{d-1}{{\bf n}'}$ changing the orientation by 
	\[
		(\sgn \sigma)^{\dim \wreathOfW{k-1}{d-1}{{\bf n}'}} = (\sgn \sigma)^{(d-1)(n_1 \cdots n_{k-1} - 1)}.
	\]
	In the step $(\star\star)$ each of the $n_k$ elements $\Sigma_i \in \wreathOfSym{k-1}{{\bf n}'}$ acts on $V_{\sigma^{-1}(i)} \in \wreathOfW{k-1}{d-1}{{\bf n}'}$ and changes the orientation by $\orient(\Sigma_i)$. 
	
	\medskip
	In total, an element $((\Sigma_1, \dots , \Sigma_{n_k});\sigma) \in \wreathOfSym{k}{\bf n}$ acts on the vector space $\wreathOfW{k}{d-1}{\bf n}$ and changes the orientation by
	\[
		(\sgn \sigma)^{d-1} \cdot (\sgn \sigma)^{(d-1)(n_1 \cdots n_{k-1} - 1)} \cdot \orient(\Sigma_1) \cdots \orient(\Sigma_{n_k}) = \orient(\Sigma_1, \dots , \Sigma_{n_k};\sigma)
	\]
	as claimed.
\end{proof}

\medskip
To evaluate the obstruction cocycle, we use the $\wreathOfSym{k}{\bf n}$-equivariant linear projection
\[
	f_k\colon \wreathOfW{k}{d}{\bf n} \longrightarrow \wreathOfW{k}{d-1}{\bf n},
\]
given by forgetting the $d$th coordinate.
It serves as a general position map and is defined inductively as follows.
\begin{compactitem}[\quad --]
	\item For $k = 1$ the map $f_1\colon W_n^{\oplus d} \to W_n^{\oplus d-1}$ forgets the $d$th coordinate.
	
	\item For $k \ge 2$ we define $f_k$ to be the $\wreathOfSym{k}{\bf n}$-equivariant map
	\[
		(f_{k-1})^{\oplus n_k} \oplus f_1: \wreathOfW{k-1}{d}{{\bf n}'}^{\oplus n_k} \oplus W_{n_k}^{\oplus d} \longrightarrow \wreathOfW{k-1}{d-1}{{\bf n}'}^{\oplus n_k} \oplus W_{n_k}^{\oplus d-1}.
	\]
\end{compactitem}
Since by construction we have $\cellWreathOfConf{1}{d}{n} = \cellConf(d, n) \subseteq W_n^{\oplus d} = \wreathOfW{1}{d}{n}$ (see \cite[Sec.\,3]{BlagojevicZiegler15}), by restriction of the domain we can speak of the $\wreathOfSym{1}{n}$-equivariant map
\[
	f_1\colon\cellWreathOfConf{1}{d}{n} \longrightarrow \wreathOfW{1}{d-1}{n}.
\]
By induction on $k \ge 1$ it follows that $\cellWreathOfConf{k}{d}{\bf n} \subseteq \wreathOfW{k}{d}{\bf n}$, so we may speak about the $\wreathOfSym{k}{\bf n}$-equivariant map
\[
	f_k\colon\cellWreathOfConf{k}{d}{\bf n} \longrightarrow  \wreathOfW{k}{d-1}{\bf n}.
\]

\medskip
\begin{lemma} \label{lem: cell orientations and c_f(e) = +1}
	Let $k \ge 1$, $d \ge 1$ and $n_1, \dots , n_k \ge 2$ be integers. 
	Then the following statements hold.
	\begin{enumerate}[\quad \rm(i)]
		\item\label{lem: cell orientations and c_f(e) = +1 -- 01} The linear map $f_k$ maps all $M_k$-cells of $\cellWreathOfConf{k}{d}{n_1, \dots , n_k}$ by a cellular homeomorphism to the same star-shaped $\wreathOfSym{k}{n_1, \dots , n_k}$-neighbourhood $\B_k \subseteq \wreathOfW{k}{d-1}{n_1, \dots , n_k}$ of the origin.

		\item\label{lem: cell orientations and c_f(e) = +1 -- 02} There exists an orientation of the $M_k$- and $(M_k-1)$-cells of the cell complex $\cellWreathOfConf{k}{d}{n_1, \dots , n_k}$ such that the cellular action of $\wreathOfSym{k}{n_1, \dots , n_k}$ on $M_k$- and $(M_k-1)$-cells changes the orientation according to the orientation function $\orient$ defined in Definition \ref{def: orientation function}.
		
		\item\label{lem: cell orientations and c_f(e) = +1 -- 03} Assuming the orientations of the cells of $\cellWreathOfConf{k}{d}{n_1, \dots , n_k}$ from part \eqref{lem: cell orientations and c_f(e) = +1 -- 02}, the obstruction cocycle $c(f_k)$ has the value $+1$ on all oriented $M_k$-cells of $\cellWreathOfConf{k}{d}{n_1, \dots , n_k}$.
		
		\item\label{lem: cell orientations and c_f(e) = +1 -- 04} Assuming the orientations of the cells of $\cellWreathOfConf{k}{d}{n_1, \dots , n_k}$ from part \eqref{lem: cell orientations and c_f(e) = +1 -- 02}, the following formula of cellular chains holds
		\[
			\partial e_k = \sum_{b \in B_k} \sgn_k(b) \cdot e(b) \in C_{M_k - 1}(\cellWreathOfConf{k}{d}{n_1, \dots , n_k}),
		\]
		where $e(b)$ denotes the $(M_k-1)$-cell in the boundary of $e_k$ indexed by $b \in B_k$ in light of Lemma \ref{lem: orbits of M and (M-1)-cells and B_k}\eqref{lem: orbits of M and (M-1)-cells and B_k -- 02} , and $\sgn_k(b) \in \{-1, +1\}$ is a sign function constant on each $\wreathOfSym{k}{n_1, \dots , n_k}$-orbit.
	\end{enumerate}
\end{lemma}
\begin{proof}
	All statements are proved by the (same) induction on $k \ge 1$. 
	The base case $k=1$ is treated in \mbox{\cite[Lem.\,4.1]{BlagojevicZiegler15}}. 
	There, the $(M_1-1)$-cells in $\partial \B_1$ are oriented such that they all appear with a sign $+1$ in the cellular boundary formula of $\B_1$. 
	
	\medskip
	Now, assume that $k \ge 2$.
	\begin{compactitem} [\quad --]
		\item We define 
		$
			\B_k := (\B_{k-1})^{\times n_k} \times \B_1
		$,
		where $\B_{k-1} \subseteq \wreathOfW{k-1}{d-1}{{\bf n}'}$ is the star-shaped neighbourhood from the induction hypothesis. Any $M_k$-cell $e$ of $\cellWreathOfConf{k}{d}{\bf n}$ equals to the product of $n_k$ maximal cells of $\cellWreathOfConf{k-1}{d}{{\bf n}'}$ and a maximal cell of $\cellConf(d, n_k)$. Therefore, by induction hypothesis and the inductive definition of the map $f_k$, it follows that $f_k$ restricted to the cell $e$ is cellular homeomorphism.
		
		\item Let $\B_k$ and the $(M_k-1)$-cells in $\partial \B_k$ be endowed with the product orientation. The map $f_k$ sends each $M_k$- and $(M_k-1)$-cell of $\cellWreathOfConf{k}{d}{\bf n}$ homeomorphically to $\B_k$ and a $(M_k-1)$-cell in the boundary of $\B_k$, respectively. Thus, we can set orientation on each $M_k$- and $(M_k-1)$-cell of $\cellWreathOfConf{k}{d}{\bf n}$ such that $f_k$ restricted to them is orientation preserving. By Lemma \ref{lem: wreathSym action on wreathW} the $\wreathOfSym{k}{\bf n}$-action on $\B_k$ changes the orientation by the sign given by orientation function $\orient$, so the same is true on the cells of $\cellWreathOfConf{k}{d}{\bf n}$.
		
		\item Since $f_k$ is orientation preserving cellular homeomorphism, we have		
		\[
			c(f_k)(e) = \deg(r \circ f_k\colon \partial e \longrightarrow \wreathOfW{k}{d-1}{\bf n} \setminus \{0\} \longrightarrow \partial \B_k) = 1,
		\]
		for each $M_k$-cell $e \subseteq \cellWreathOfConf{k}{d}{\bf n}$.
		
		\item Since $\cellWreathOfConf{k}{d}{\bf n}$ is regular, we need to show only that $\sgn(b) = \sgn(b')$ for any $b,b' \in B_k$ in the same orbit. By the boundary of the cross-product formula for $e_k = (e_{k-1})^{\times n_k} \times \check{c}(\id, {\bf d})$, we have the following equality of cellular $(M_k-1)$-chains:
		\begin{equation*}
			\partial e_k = \sum_{i=1}^{n_k} (-1)^{(i-1)M_{k-1}} e_{k-1}^{i-1} \times \partial e_{k-1} \times e_{k-1}^{n_k-i} \times \check{c}(\id, {\bf d}) + (-1)^{n_kM_{k-1}}e_{k-1}^{n_k} \times \partial \check{c}(\id, {\bf d}).
		\end{equation*}
		By the induction hypothesis, we have
		\[
			\partial e_{k-1} = \sum_{c \in B_{k-1}} \sgn_{k-1}(c) \cdot e(c) \qquad \text{ and } \qquad \partial\check{c}(\id, {\bf d}) = \sum_{a \in B_1} \sgn_1(a) \check{c}(a).
		\]
		Moreover, we have $\sgn_{k-1}(c) = \sgn_{k-1}(c')$ if $c,c' \in B_{k-1}$ are in the same $\wreathOfSym{k-1}{{\bf n}'}$-orbit and $\sgn_1(a) = \sgn_1(a')$ if $a,a' \in B_{1}$ are in the same $\sym_{n_k}$-orbit. Similarly to the proof of Lemma \ref{lem: orbits of M and (M-1)-cells and B_k}, boundary cells of $e_k$ are divided into two types.
		\begin{compactenum}[\quad (1)]
			\item Cells of the first type are of the form 
			\[
				(e_{k-1})^{\times n_k} \times e(c) \times (e_{k-1})^{\times n_k-i} \times \check{c}(\id, {\bf d})
			\]
			for $1 \le i \le n_k$ and $c \in B_{k-1}$. By part (iii) of Lemma \ref{lem: orbits of M and (M-1)-cells and B_k} or its proof, it is seen that such a cell receives a label
			\[
				(c,i) \in B_{k-1} \times [n_k] \subseteq B_k.
			\]
			Moreover, two cells labeled by $(c,i), (c', i') \in B_{k-1} \times [n_k] \subseteq B_k$ are in the same $\wreathOfSym{k}{\bf n}$-orbit if and only if $i = i'$ and $c, c' \in B_{k-1}$ are in the same $\wreathOfSym{k-1}{{\bf n}'}$-orbit. Therefore, by the inductive sign formula for the cells of the first type
			\[
				\sgn_k(c,i) = (-1)^{(i-1)M_{k-1}} \cdot \sgn_{k-1}(c)
			\]
			the claim follows for the boundary cells of the first type.
			
			\item Cells of the second type are of the form 
			\[
				(e_{k-1})^{\times n_k} \times \check{c}(a)
			\]
			for $a \in B_{1}$. Again by part (iii) of Lemma \ref{lem: orbits of M and (M-1)-cells and B_k} or its proof, it is seen that such a cell receives a label $a \in B_1 \subseteq B_k$. Moreover, two cells labeled by $a, a' \in B_1 \subseteq B_k$ are in the same $\wreathOfSym{k}{\bf n}$-orbit if and only if $a,a' \in B_1$ are in the same $\sym_{n_k}$-orbit. Therefore, by the inductive sign formula for the cells of the second type
			\[
				\sgn_k(a) = (-1)^{n_kM_{k-1}} \cdot \sgn_1(a)
			\]
			the claim follows for the boundary cells of the second type.
		\end{compactenum}
		Finally, by part (iii) of Lemma \ref{lem: orbits of M and (M-1)-cells and B_k} or its proof, no cell of the first type is in the same orbit as a cell of the second type, so the proof is completed. 	
	\end{compactitem}
\end{proof}

\medskip
\subsection{When is the obstruction cocycle a coboundary?} \label{subsection: when is the obstruction cocycle a coboundary?}

In this section we discuss when the computed cocycle $c(f_k)$ is a coboundary, and consequently complete the proof of Theorem \ref{theorem: main}. 
The orientation of cells of the complex $\cellWreathOfConf{k}{d}{\bf n}$ is understood to be the one from Lemma \ref{lem: cell orientations and c_f(e) = +1}\eqref{lem: cell orientations and c_f(e) = +1 -- 02}. The following lemma is a generalisation of the $k = 1$ case treated in \cite[Lem.\,4.2]{BlagojevicZiegler15}.

\medskip
\begin{lemma} \label{lem: coboundary of (M-1)-cochain}
	Let $k \ge 1$, $d \ge 1$ and $n_1, \dots , n_k \ge 2$ be integers. Then the following statements hold.
	\begin{compactenum} [\quad \rm (i)]
		\item\label{lem: coboundary of (M-1)-cochain -- 01} The value of the coboundary $\delta w$ of any equivariant cellular cochain 
		\[
			w \in C^{M_k-1}_{\wreathOfSym{k}{n_1, \dots , n_k}}(\cellWreathOfConf{k}{d}{n_1, \dots , n_k};\ZZ_k)
		\]
		is the same on each $M_k$-cell of $\cellWreathOfConf{k}{d}{n_1, \dots , n_k}$ and is equal to the $\Z$-linear combination of binomial coefficients
		\begin{equation} \label{eq: binomial coefficients}
			\binom{n_1}{1}, \dots , \binom{n_1}{n_1 - 1}, \dots , \binom{n_k}{1}, \dots , \binom{n_k}{n_k - 1}.	
		\end{equation}

		\item\label{lem: coboundary of (M-1)-cochain -- 02} For any $\Z$-linear combination of binomial coefficients \eqref{eq: binomial coefficients}, there exists a cochain whose coboundary takes precisely that value on all $M_k$-cells of $\cellWreathOfConf{k}{d}{n_1, \dots , n_k}$.
	\end{compactenum} 
\end{lemma}
\begin{proof}
	Let
	\[
		w \in C^{M_k-1}_{\wreathOfSym{k}{\bf n}}(\cellWreathOfConf{k}{d}{\bf n};\ZZ_k) = \hom_{\wreathOfSym{k}{\bf n}}(C_{M_k - 1}(\cellWreathOfConf{k}{d}{\bf n}), \ZZ_k)
	\]
	be a cellular cochain. 
	The group $\wreathOfSym{k}{\bf n}$:
	\begin{compactitem}[\quad --]
		\item changes the orientation of any $M_k$- or $(M_k-1)$-cell according to the orientation function $\orient$ (consult Lemma \ref{lem: cell orientations and c_f(e) = +1} (ii)), and
		\item  acts on $\ZZ_k$ by multiplication with the value of $\orient$ (see \eqref{eq: wreathSym-module structure on ZZ}).
	\end{compactitem}
	Therefore, an $\wreathOfSym{k}{\bf n}$-module morphism of the form
	\[
		C_{M_k-1}(\cellWreathOfConf{k}{d}{\bf n}) \longrightarrow \ZZ_k ~~~ \text{ or} ~~~~ C_{M_k}(\cellWreathOfConf{k}{d}{\bf n}) \longrightarrow \ZZ_k
	\]
	is constant on each $\wreathOfSym{k}{\bf n}$-orbit.
	
	\medskip
	In particular, since $M_k$-cells form a single $\wreathOfSym{k}{\bf n}$-orbit by Lemma \ref{lem: orbits of M and (M-1)-cells and B_k}\eqref{lem: orbits of M and (M-1)-cells and B_k -- 01}, it follows that the coboundary $\delta w$ has the same value on all $M_k$-cells. 
	Therefore, we restrict our attention to the orbit representative $M_k$-cell $e_k$ introduced in Definition \ref{def: orbit representative maximal cell} and the value $(\delta w)(e_k) = w(\partial e_k)$.
	
	\medskip
	From Lemma \ref{lem: cell orientations and c_f(e) = +1}\eqref{lem: cell orientations and c_f(e) = +1 -- 04} we have that 
	\[
		\partial e_k = \sum_{b \in B_k} \sgn(b) \cdot e(b) \in C_{M_k - 1}(\cellWreathOfConf{k}{d}{\bf n}),
	\]
	with the sign function $\sgn$ being constant on $\wreathOfSym{k}{\bf n}$-orbits of $(M_k-1)$-cells. 
	According to Lemma \ref{lem: orbits of M and (M-1)-cells and B_k}\eqref{lem: orbits of M and (M-1)-cells and B_k -- 03}, the $\wreathOfSym{k}{\bf n}$-orbit stratification of the index set $B_k$ is given by
	\[
		\Big\{\binom{[n_i]}{m} \times \{{\bf j}\}:~ 1 \le i \le k,~ 1 \le m \le n_i - 1,~ {\bf j} \in [n_{i+1}] \times \dots \times [n_{k}] \Big\},
	\]
	and we may denote by $w(i, m; {\bf j})$ and $\sgn(i, m; {\bf j})$ the values of $w$ and $\sgn$ on the orbit containing the corresponding stratum. With this notation in hand, we have
	\begin{equation} \label{eq: (delta w)(e_k)}
		(\delta w)(e_k) = w(\partial e_k) = \sum_{1 \le i \le k} ~\sum_{1 \le m \le n_i - 1} ~ \sum_{{\bf j} \in [n_{i+1}] \times \dots \times [n_{k}]} \binom{n_i}{m}\sgn(i,m;{\bf j}) w(i,m;{\bf j}),
	\end{equation}
	which proves part \eqref{lem: coboundary of (M-1)-cochain -- 01} of the lemma. 
	
	\medskip
	To prove part \eqref{lem: coboundary of (M-1)-cochain -- 02}, let us show that any $\Z$-linear combination	
	\[
		\sum_{1 \le i \le k} ~\sum_{1 \le m \le n_i - 1} x_{i,m} \cdot \binom{n_i}{m}
	\]
	is equal to the value $(\delta w)(e_k)$, for some cellular $(M_k-1)$-cochain $w$. The value of $w$ can be chosen independently on each orbit. For example, one can set $w$ to be the equivariant extension of the map given on boundary cells of $e_k$ as
	\[
		w(e(b)) := \begin{cases}
			x_{i, m} \cdot \sgn(i,m;(1, \dots , 1)), & b \in \binom{[n_i]}{m} \times \{(1, \dots , 1)\} \text{ for some } i \in [k], m \in [n_i-1],\\
			0, & \text{otherwise},
		\end{cases}
	\]
	for each $b \in B_k$. Indeed, by \eqref{eq: (delta w)(e_k)} we have
	\[
		(\delta w)(e_k) = \sum_{1 \le i \le k} ~ \sum_{1 \le m \le n_i-1} \binom{n_i}{m}\sgn(i,m;{\bf 1})w(i,m;{\bf 1}) = \sum_{1 \le i \le k} ~ \sum_{1 \le m \le n_i-1} \binom{n_i}{m} x_{i,m},
	\]
	where ${\bf 1} := (1, \dots ,1) \in [n_{i+1}] \times \dots \times [n_k]$ is the notation in each summand.
\end{proof}

\medskip
The next elementary consequence of the Ram's result \cite{Ram09} is used in the proof of Theorem \ref{theorem: main}.

\medskip
\begin{lemma} \label{lem: gcd of binoms}
	Let $k \ge 1$, $d \ge 1$ and $n_1, \dots , n_k \ge 2$ be integers. 
	Then we have that
	\begin{equation*}
		\gcd\Big\{\binom{n_i}{1}, \dots , \binom{n_i}{n_i - 1}:~ 1 \le i \le k\Big\} = \begin{cases}
			p, & \text{if all } n_i \text{ are powers of the same prime } p,\\
			1, & \text{otherwise.}
		\end{cases}
	\end{equation*}
\end{lemma}
\begin{proof}
	By the work of Ram's result \cite{Ram09} it follows that for each $1 \le i \le k$ we have
	\begin{equation*}
		N_i := \gcd\Big\{\binom{n_i}{1}, \dots , \binom{n_i}{n_i - 1}\Big\} = \begin{cases}
			p, & \text{if } n_i \text{ is a power of a prime } p,\\
			1, & \text{otherwise.}
		\end{cases}
	\end{equation*}
	Therefore, we have $\gcd(N_1, \dots , N_k) \neq 1$ if and only if $n_1, \dots n_k$ are all powers of the same prime $p$, in which case we have $\gcd(N_1, \dots , N_k) = p$.
\end{proof}

\medskip
Finally, we are ready to give a proof of Theorem \ref{theorem: main}.

\medskip
\begin{proof}[Proof of Theorem \ref{theorem: main}] 
By \cite[Sec.\,II.3]{tomDieck:TransformationGroups} the equivariant map exists if and only if the cohomology class ${\frak o} = [c(f_k)]$ vanishes. 
This happens if and only if $c(f_k)$ is the coboundary of some equivariant cellular cochain in $C^{M-1}_{\wreathOfSym{k}{\bf n}}(\cellWreathOfConf{k}{d}{\bf n};\ZZ_k)$. 
By Lemmas \ref{lem: cell orientations and c_f(e) = +1} and \ref{lem: coboundary of (M-1)-cochain} this happens if and only if there exists a linear combination of binomial coefficients
	\[
		\binom{n_1}{1}, \dots , \binom{n_1}{n_1 - 1}, \dots , \binom{n_k}{1}, \dots , \binom{n_k}{n_k - 1}
	\]
	which is equal to $1$. This equivalent to the fact that the greatest common divisor of all such binomial coefficients equals $1$, which by Lemma \ref{lem: gcd of binoms} happens if and only if $n_1, \dots , n_k$ are not all powers of the same prime number.
\end{proof}

Moreover, we are able to compute the top equivariant cohomology of the wreath product of configuration spaces with coefficients in the module $\ZZ_k$.

\begin{corollary}
	Let $d\ge 2$, $k \ge 1$ and $n_1, \dots, n_k \ge 2$ be integers. Then we have
	\[
		H^{M_k}_{\wreathOfSym{k}{n_1, \dots, n_k}}(\wreathOfConf{k}{d}{n_1, \dots, n_k};\ZZ_k) = \langle [c(f_k)] \rangle \cong \begin{cases}
			\Z/p & n_1, \dots , n_k \textrm{ are powers of a prime number }p\\
			0 & \textrm{otherwise.}
		\end{cases}
	\]
\end{corollary}
\begin{proof}
	The cellular complex $\cellWreathOfConf{k}{d}{\bf n}$ has a single $\wreathOfSym{k}{\bf n}$-orbit of maximal cells, so
	\[
		C^{M_k}_{\wreathOfSym{k}{\bf n}}(\cellWreathOfConf{k}{d}{\bf n};\ZZ_k) \cong \Z\langle c(f_k) \rangle.
	\]
	By Lemma \ref{lem: coboundary of (M-1)-cochain} it follows that
	\[
		H^{M_k}_{\wreathOfSym{k}{\bf n}}(\cellWreathOfConf{k}{d}{\bf n};\ZZ_k) \cong \Z\langle c(f_k) \rangle/ \Z \langle N \cdot c(f_k) \rangle,
	\]
	where $N := \gcd\{ \binom{n_i}{1}, \dots, \binom{n_i}{n_i-1}:~1 \le i \le k\}$, so the claim follows by Lemma \ref{lem: gcd of binoms}.
\end{proof}

As another consequence of the cellular model, we obtain the following extension of \cite[Cor.~4.7]{BlagojevicZiegler15}, following analogous proof.

\begin{corollary}
	Let $k \ge 1$, $d\ge 2$ and $n_1, \dots, n_k \ge 2$ be integers, and $G \coloneqq \Z_{n_1} \wr \dots \wr \Z_{n_k}$. Let $X$ be a free Hausdorff $G$-space and $f \colon X \longrightarrow \R^d$ a continuous map. If $X$ is $(d-1)(n-1)$-connected, where $n \coloneqq n_1 \dots n_k$, there there exist $x \in X$ and nontrivial $g \in G$ such that
	\[
		f(x) = f(g\cdot x).
	\]
\end{corollary}
\begin{proof}
	Assume the contrary. Then, due to free action on $X$, it follows by induction on $k$ that there exists $G$-equivariant map
	\[
		X \longrightarrow \wreathOfConf{k}{d}{n_1, \dots, n_k},\qquad x \longmapsto(g\cdot x)_{g \in G},
	\]
	where we consider $G \subseteq \wreathOfSym{k}{n_1, \dots, n_k}$. 
	
	\medskip
	Post-composing this with the equivariant retraction $\wreathOfConf{k}{d}{n_1, \dots, n_k} \simeq \cellWreathOfConf{k}{d}{n_1, \dots, n_k}$, we obtain a $G$-equivariant map
	\[
		X \longrightarrow\cellWreathOfConf{k}{d}{n_1, \dots, n_k}
	\]
	from an $(n-1)(d-1)$-connected $G$-space into an $(n-1)(d-1)$-dimensoinal free $G$-space, which contradicts Dold's theorem \cite{Dold83}.
\end{proof}

\subsection{Another proof of the non-existence}

In this section we give another proof of non-existence of an $\wreathOfSym{k}{n_1, \dots, n_k}$-equivariant map
\begin{equation} \label{eq: non-existence map}
	\wreathOfConf{k}{d}{n_1, \dots, n_k} \longrightarrow S(\wreathOfW{k}{d-1}{n_1, \dots, n_k})
\end{equation}
when $n_1, \dots, n_k$ are powers of the same prime. We are thankful to the anonymous referee for pointing out the following argument.

\medskip
Let $n_1=p^{m_1}, \dots, n_k = p^{m_k}$ be powers of a prime number $p$ and let $m := m_1 + \cdots + m_k$.
First, let us notice that we can consider $\wreathOfSym{m}{p, \dots, p} = \aut(P_m(p, \dots, p))$ as a subgroup of $\wreathOfSym{k}{p^{m_1}, \dots, p^{m_k}} = \aut(P_k(p^{m_1}, \dots, p^{m_k}))$, since there is a natural identification of the sets of $p^m$ leaves of the two posets $P_m(p, \dots, p)$ and $P_k(p^{m_1}, \dots, p^{m_k})$ on which the groups are acting.

\medskip
By induction on $k$ and using the map from Lemma \ref{lem: wreathConf to conf} induced by the little cubes operad, there is an $\wreathOfSym{m}{p, \dots, p}$-equivariant map 
\begin{equation} \label{eq: map from Cm(p,...,p) to Ck(n)}
	\wreathOfConf{m}{d}{p, \dots, p} ~ \longrightarrow~ \wreathOfConf{k}{d}{p^{m_1}, \dots, p^{m_k}}.
\end{equation}
The next key ingredient is the following non-existence result from the PhD thesis of Pali\'c \cite[Prop.\,5.4.1]{Palic18}.

\begin{proposition} \label{prop: non-existence by nevena}
	Let $k, d \ge 1$ be integers and $p \ge 2$ be a prime number. Then, there is no $\wreathOfSym{m}{p, \dots, p}$-equivariant map of the form
\begin{equation*}
	\wreathOfConf{m}{d}{p, \dots, p} ~ \longrightarrow~ S(\wreathOfW{m}{d-1}{p, \dots, p}).
\end{equation*}
\end{proposition}

Since $\wreathOfW{m}{d-1}{p, \dots, p}$ and $\wreathOfW{k}{d-1}{p^{m_1}, \dots, p^{m_k}}$ are isomorphic as vector spaces, and after considering $\wreathOfSym{m}{p, \dots, p}$ as a subgroup of $\wreathOfSym{k}{p^{m_1}, \dots, p^{m_k}}$ as before, we see that the vector spaces are moreover isomorphic as $\wreathOfSym{m}{p, \dots, p}$-modules.

Thus, the existence of an $\wreathOfSym{k}{p^{m_1}, \dots, p^{m_k}}$-equivariant map \eqref{eq: non-existence map} would, together with the $\wreathOfSym{m}{p, \dots, p}$-equivariant map \eqref{eq: map from Cm(p,...,p) to Ck(n)}, imply the existence of an $\wreathOfSym{m}{p, \dots, p}$-equivariant map 
\begin{equation*}
	\wreathOfConf{m}{d}{p, \dots, p} ~ \longrightarrow~ S(\wreathOfW{m}{d-1}{p, \dots, p}),
\end{equation*}
which is prohibited by the above proposition.

\medskip
Let us denote by $\wreathOfSym{m}{p, \dots, p}^{(p)}$ the $p$-Sylow subgroup of $\wreathOfSym{m}{p, \dots, p}$. For more details on Sylow subgroups, see for example \cite[Sec.\,4.5]{DummitFoote04}.
What is proved in Pali\'c's thesis \cite[Prop.5.4.1]{Palic18} is actually a little bit stronger than Proposition \ref{prop: non-existence by nevena}. 
Namely, it was shown via equivariant obstruction theory that there is no $\wreathOfSym{m}{p, \dots, p}^{(p)}$-equivariant map of the same form.

\medskip
\section{Existence in Theorem \ref{theorem: main}: the little cubes operad} 
\label{section: existence via operad}
\medskip

In this section we give a short proof of the existence part of Theorem \ref{theorem: main} using the little cubes operad structural map. 
As before, let us set ${\bf n} := (n_1, \dots ,n_k)$.
 Our goal is to show the existence of an $\wreathOfSym{k}{\bf n}$-equivariant map
\[
	\wreathOfConf{k}{d}{\bf n} \longrightarrow S(\wreathOfW{k}{d-1}{\bf n})
\]
in the case when $n_1, \dots , n_k$ are not all powers of the same prime number.

\medskip
In the remainder of the section, we will consider $\wreathOfSym{k}{\bf n}$ as the subgroup of the symmetric group $\sym_n$, where $n := n_1 \cdots n_k$. 
Indeed, $\wreathOfSym{k}{\bf n}$ is the automorphism group of the poset $P_k({\bf n})$ by Lemma \ref{lem: aut = wreathSym}. 
In particular, it permutes the $n$ level $k$ elements, i.e., leaves, which have the symmetry of the group $\sym_n$.
Note that this description specifies the inclusion of $\wreathOfSym{k}{\bf n}$ into  $\sym_n$.

\medskip
Let $(\littleCubeOperad{d}{n})_{n \ge 1}$ denote the {\em little cubes operad} as introduced by May \cite{May72}. See also \cite[Sec.\,7.3]{BCCLZ21} for notation. 
Each $\littleCubeOperad{d}{n}$ is a free $\sym_n$-space. The operad comes with the {\em structural map}
\[
	\mu\colon (\littleCubeOperad{d}{m_1} \times \dots \times \littleCubeOperad{d}{m_k}) \times \littleCubeOperad{d}{k} \longrightarrow \littleCubeOperad{d}{m_1 + \dots + m_k},
\]
for any integers $k \ge 1$ and $m_1, \dots ,m_k \ge 1$. See Figure \ref{fig 05} for illustration.

\begin{figure}[h]
\begin{tikzpicture}[scale = 0.8]
	\node at (-3,0) {$\Bigg($};
	
	\draw[dashed] (-2.5,-1.5) -- (0.5,-1.5);
	\draw[dashed] (0.5,-1.5) -- (0.5,1.5);
	\draw[dashed] (0.5,1.5) -- (-2.5,1.5);
	\draw[dashed] (-2.5,1.5) -- (-2.5,-1.5);
	
	\draw (-2.5,0.25) ++ (1,0) -- ++(1,0) -- ++(0,1) -- ++(-1,0) -- ++(0,-1);
	\draw (-2.5,-1.25) ++ (1,0) -- ++(1,0) -- ++(0,1) -- ++(-1,0) -- ++(0,-1);
	
	\node at (0.75,-0.5) {$,$};
	
	\draw[dashed] (1,-1.5) -- (4,-1.5);
	\draw[dashed] (4,-1.5) -- (4,1.5);
	\draw[dashed] (4,1.5) -- (1,1.5);
	\draw[dashed] (1,1.5) -- (1,-1.5);
	
	\draw (1,0.25) ++ (1,0) -- ++(1,0) -- ++(0,1) -- ++(-1,0) -- ++(0,-1);
	\draw (1.75,-1) ++ (1,0) -- ++(1,0) -- ++(0,1) -- ++(-1,0) -- ++(0,-1);
	\draw (0.25,-1.25) ++ (1,0) -- ++(1,0) -- ++(0,1) -- ++(-1,0) -- ++(0,-1);
	
	\node at (4.25,-0.5) {$ ;$};
	
	\draw[dashed] (4.5,-1.5) -- (7.5,-1.5);
	\draw[dashed] (7.5,-1.5) -- (7.5,1.5);
	\draw[dashed] (7.5,1.5) -- (4.5,1.5);
	\draw[dashed] (4.5,1.5) -- (4.5,-1.5);
	
	\draw (5.25,0.25) ++ (1,0) -- ++(1,0) -- ++(0,1) -- ++(-1,0) -- ++(0,-1);
	\draw (3.75,-1.25) ++ (1,0) -- ++(1,0) -- ++(0,1) -- ++(-1,0) -- ++(0,-1);
	
	\node at (8,0) {$\Bigg )$};
	\node at (8.75,0) {$\longmapsto$};
	
	\draw[dashed] (9.5,-1.5) -- (12.5,-1.5);
	\draw[dashed] (12.5,-1.5) -- (12.5,1.5);
	\draw[dashed] (12.5,1.5) -- (9.5,1.5);
	\draw[dashed] (9.5,1.5) -- (9.5,-1.5);
	
	\draw[dashed] (10.25,0.25) ++ (1,0) -- ++(1,0) -- ++(0,1) -- ++(-1,0) -- ++(0,-1);
	\draw[dashed] (8.75,-1.25) ++ (1,0) -- ++(1,0) -- ++(0,1) -- ++(-1,0) -- ++(0,-1);
	
	\draw (9.1,-0.67) ++ (1,0) -- ++(0.33,0) -- ++(0,0.33) -- ++(-0.33,0) -- ++(0,-0.33);
	\draw (9.1,-1.15) ++ (1,0) -- ++(0.33,0) -- ++(0,0.33) -- ++(-0.33,0) -- ++(0,-0.33);
	
	\draw (10.6,0.85) ++ (1,0) -- ++(0.33,0) -- ++(0,0.33) -- ++(-0.33,0) -- ++(0,-0.33);
	\draw (10.8,0.45) ++ (1,0) -- ++(0.33,0) -- ++(0,0.33) -- ++(-0.33,0) -- ++(0,-0.33);
	\draw (10.35,0.35) ++ (1,0) -- ++(0.33,0) -- ++(0,0.33) -- ++(-0.33,0) -- ++(0,-0.33);
	
\end{tikzpicture}

\caption{\small The structural map $\mu: (\littleCubeOperad{2}{2} \times \littleCubeOperad{2}{3}) \times \littleCubeOperad{2}{2} \to \littleCubeOperad{2}{5}$ of the little cubes operad.}
\label{fig 05}	
\end{figure}

Moreover, we have the following fact \cite[Thm.\,4.8]{May72}. For illustration see Figure \ref{fig 06}.

\medskip
\begin{lemma} \label{lem: littleCube and F are homotopy equivalent}
	Let $d \ge 1$ and $n \ge 1$ be integers. Then, the space $\littleCubeOperad{d}{n}$ is $\sym_n$-homotopy equivalent to the configuration space $\conf(\R^d,n)$.
\end{lemma}

\medskip
The special case  of the structural map when $m_1 = \dots = m_k$ is useful for us, as explained in the following lemma. Consult also \cite[Lem.\,7.2]{BCCLZ21}.

\medskip
\begin{lemma} \label{lem: mu is equivariant}
	Let $d \ge 1$ and $n \ge 1$ be integers. The structural map
\[
	\mu\colon \littleCubeOperad{d}{m}^{\times k} \times \littleCubeOperad{d}{k} \longrightarrow \littleCubeOperad{d}{mk}
\]
is $((\sym_m)^{\times k } \rtimes \sym_k)$-equivariant. The action of $(\sym_m)^{\times k } \rtimes \sym_k$ on the domain is assumed to be the wreath product action from Definition \ref{def: wreath product action}, and the action on the codomain is given by restriction $(\sym_m)^{\times k } \rtimes \sym_k = \wreathOfSym{2}{m,k} \subseteq \sym_{mk}$.
\end{lemma}

\begin{figure}[h]
\center
\begin{tikzpicture}[scale = 0.8]
		\draw[dashed] (-2.5,0) -- (2.5,0);
		\draw[dashed] (2.5,0) -- (2.5,5);
		\draw[dashed] (2.5,5) -- (-2.5,5);
		\draw[dashed] (-2.5,5) -- (-2.5,0);
		
		\draw (-2.5,0.5) ++ (1,0) -- ++(1,0) -- ++(0,1) -- ++(-1,0) -- ++(0,-1);
		\node at (-0.75,1.25) {$x_2$};
		\node at (-1,1) {$\cdot$};
		\node at (-2,1.5) {$C_2$};
		
		\draw (-0.5,1.5) ++ (1,0) -- ++(1,0) -- ++(0,1) -- ++(-1,0) -- ++(0,-1);
		\node at (1.25,2.25) {$x_3$};
		\node at (1,2) {$\cdot$};
		\node at (2,2.5) {$C_3$};
		
		\draw (-1.5,3.5) ++ (1,0) -- ++(1,0) -- ++(0,1) -- ++(-1,0) -- ++(0,-1);
		\node at (0.25,4.25) {$x_1$};
		\node at (0,4) {$\cdot$};
		\node at (1,4.5) {$C_1$};
		
		\node at (3.5,2.5) {$\longmapsto$};
		
		\node at (5.25,1.25) {$x_2$};
		\node at (5,1) {$\cdot$};
	
		\node at (7.25,2.25) {$x_3$};
		\node at (7,2) {$\cdot$};
		
		\node at (6.25,4.25) {$x_1$};
		\node at (6,4) {$\cdot$};
\end{tikzpicture}
\caption{\small Homotopy equivalence $\littleCubeOperad{2}{3} \to \conf(\R^2,3)$ mapping $(C_1, C_2, C_3)$ to the centers $(x_1, x_2, x_3)$.}
\label{fig 06}
\end{figure}

\medskip
Now, we prove the following auxiliar result.

\medskip
\begin{lemma} \label{lem: wreathConf to conf}
	Let $k \ge 1$, $d \ge 1$ and $n_1, \dots , n_k \ge 2$ be integers. 
	There exists an $\wreathOfSym{k}{n_1, \dots , n_k}$-equivariant map
	\[
		\wreathOfConf{k}{d}{n_1, \dots , n_k} \longrightarrow \conf(\R^d, n),
	\]
	where $n := n_1 \cdots  n_k$ and the action on the codomain is the restriction action $\wreathOfSym{k}{n_1, \dots , n_k} \subseteq \sym_{n}$.
\end{lemma}
\begin{proof}
	We will show the existence of an $\wreathOfSym{k}{n_1, \dots , n_k}$-equivariant map
	\[
		\gamma_{k}\colon\wreathOfConf{k}{d}{n_1, \dots , n_k} \longrightarrow  \conf(\R^d, n),
	\]
	by induction on $k \ge 1$. 
	In the base case $k=1$, the map $\gamma_1$ is the equality.
	
	\medskip
	Assume $k \ge 2$. 
	As before, let $n' := n_1 \cdots n_{k-1}$ and ${\bf n}' := (n_1, \dots ,n_{k-1})$. 
	The map $\gamma_k$ is defined by the following diagram
	\begin{equation*}
		\begin{tikzcd}
			C_{k-1}(d; {\bf n}')^{\times n_k} \times F(\R^d, n_k) \arrow[rr, "(\gamma_{k-1})^{\times  n_k} \times \id"] & & F(\R^d, n')^{\times n_k} \times F (\R^d, n_k) & F(\R^d, n)\\
			 & &  \littleCubeOperad{d}{n'}^{\times n_k} \times \littleCubeOperad{d}{n_k} \arrow[u, "\simeq"] \arrow[r, "\mu"] & \littleCubeOperad{d}{n} \arrow[u, "\simeq"].
		\end{tikzcd}
	\end{equation*}
	The group $\wreathOfSym{k}{\bf n}$ acts on $F(\R^d, n')^{\times n_k} \times F (\R^d, n_k)$ by the wreath product action (see Definition \ref{def: wreath product action}), so the top horizontal map is $\wreathOfSym{k}{\bf n}$-equivariant by the induction hypothesis. 
	Both vertical maps are $\wreathOfSym{k}{\bf n}$-equivariant by Lemma \ref{lem: littleCube and F are homotopy equivalent}. 
	Indeed, the left one is the product of equivariant homotopy equivalences, hence its homotopy inverse is equivariant as well. Finally, the bottom horizontal map is $\wreathOfSym{k}{\bf n}$-equivariant by Lemma \ref{lem: mu is equivariant}.
\end{proof}

\medskip
Next, we prove an additional auxiliar result.

\medskip
\begin{lemma} \label{lem: Wn to wreathW}
	Let $k \ge 1$, $d \ge 1$ and $n_1, \dots , n_k \ge 2$ be integers. Then, there exists an $\wreathOfSym{k}{n_1, \dots , n_k}$-module isomorphism
	\begin{equation*}
			W_n^{\oplus d-1} \xrightarrow{~ \cong ~} \wreathOfW{k}{d-1}{n_1, \dots , n_k},
	\end{equation*}
	where $n := n_1 \cdots  n_k$ and the action on the domain is the restriction action $\wreathOfSym{k}{n_1, \dots , n_k} \subseteq \sym_{n}$.
\end{lemma}
\begin{proof}
	Since we have $\wreathOfW{k}{d-1}{\bf n} \cong_{\wreathOfSym{k}{\bf n}} \wreathOfW{k}{1}{\bf n}^{\oplus d-1}$, it is enough to show there exists an $\wreathOfSym{k}{\bf n}$-equivariant monomorphism
	\begin{equation*}
			\mu_k: W_n ~\xrightarrow{~ \cong ~} ~\wreathOfW{k}{1}{\bf n}.
	\end{equation*}
	We will prove this by induction on $k \ge 1$. 
	In the base case $k = 1$ we set $\mu_1$ to be the identity. 
	
	\medskip
	Assume $k \ge 2$. 
	Let ${\bf n}' := (n_1, \dots , n_{k-1})$ and $n' = n_1 \cdots  n_{k-1}$. It is enough to show there exists an $\wreathOfSym{k}{\bf n}$-equivariant isomorphism of the form
	\begin{equation} \label{eq: inductive step map for W}
			W_n ~\xrightarrow{~ \cong ~}~ (W_{n'})^{\oplus n_k} \oplus W_{n_k},
	\end{equation}
	where $\wreathOfSym{k}{\bf n} = \wreathOfSym{k-1}{{\bf n}'}^{n_k} \rtimes \sym_{n_k}$ acts on the codomain by the wreath product action from Definition \ref{def: wreath product action} and $\wreathOfSym{k-1}{{\bf n}'}$ acts on $W_{n'}$ by the restriction $\wreathOfSym{k-1}{{\bf n}'} \subseteq \sym_{n'}$. 
	Indeed, then by the induction hypothesis the desired map can be set to be the $\wreathOfSym{k}{\bf n}$-equivariant composition
	\begin{equation*}
			W_n~ \xrightarrow{~\eqref{eq: inductive step map for W}} ~ W_{n'}^{\oplus n_k} \oplus W_{n_k}~  \xrightarrow{\mu_{k-1}^{\oplus n_k} \oplus \id} ~ \wreathOfW{k-1}{1}{{\bf n}'}^{\oplus n_k} \oplus W_{n_k} ~=~ \wreathOfW{k}{1}{{\bf n}}.
	\end{equation*}
	The map \eqref{eq: inductive step map for W} can be constructed as follows. 
	Let 
	\[
		w = (w_1, \dots w_{n_k}) \in W_n \subseteq (\R^{n'})^{n_k},
	\]
	where $w_1 + \dots + w_{n_k} = 0$. 
	Furthermore, let us denote
	\begin{compactitem}[\quad --]
		\item for each $1 \le i \le n$ by $\overline{w}_i \in \R$ the average of the $n'$ coordinates of the vector $w_i \in \R^{n'}$, and
		
		\item by $\overline{w} \in \R^{n_k}$ the average of the $n_k$ coordinates of the vector $(\overline{w}_1, \dots , \overline{w}_{n_k}) \in \R^{n_k}$. 
	\end{compactitem}
	With this notation, we have that
	\[
		w_i - \overline{w}_i \cdot (1, \dots , 1) \in W_{n'}
	\]
	for each $1 \le i \le n_k$, as well as 
	\[
		(\overline{w}_1, \dots , \overline{w}_{n_k}) - \overline{w}\cdot(1, \dots ,1) \in W_{n_k}.
	\] 
	Finally, we set the map \eqref{eq: inductive step map for W} to be such that it sends $w \in W_n$ to the vector
	\[
		\Big( w_i - \overline{w}_i\cdot(1, \dots , 1)\Big)_{i=1}^{n_k} \oplus \Big((\overline{w}_1, \dots , \overline{w}_{n_k}) - \overline{w}\cdot(1, \dots ,1) \Big) \in (W_{n'})^{\oplus n_k} \oplus W_{n_k}.
	\]
	This map is injective, and hence an isomorphism due to dimension reasons. Moreover, the restriction action $\wreathOfSym{k}{\bf n} \subseteq \sym_n$ on $W_n$ and the wreath product action of $\wreathOfSym{k-1}{{\bf n}'}^{n_k} \rtimes \sym_{n_k}$ on $(W_{n'})^{\oplus n_k} \oplus W_{n_k}$ make it $\wreathOfSym{k}{\bf n}$-equivariant.
\end{proof}

\medskip
Finally, we give the second proof of the existence part of Theorem \ref{theorem: main}.

\medskip
\begin{corollary}[Existence in Thm.\,\ref{theorem: main}] \label{cor: existence via operad}
	Let $k \ge 1$, $d \ge 1$ and $n_1, \dots , n_k \ge 2$ be integers. 
	Assume that $n_1, \dots , n_k$ are not all powers of the same prime number. Then, there exists an $\wreathOfSym{k}{n_1, \dots , n_k}$-equivariant map 
	\[
		\wreathOfConf{k}{d}{n_1, \dots , n_k} ~ \longrightarrow  ~ S(\wreathOfW{k}{d-1}{n_1, \dots , n_k}).
	\]
\end{corollary}
\begin{proof}
	If $n_1, \dots , n_k$ are not all powers of the same prime number, in particular $n := n_1 \cdots n_k$ is not a power of a prime number, hence by \cite[Thm.\,1.2]{BlagojevicZiegler15} there exists an $\sym_n$-equivariant map 
	\[
		\conf(\R^d, n) ~ \longrightarrow ~ W_n^{\oplus d-1} \setminus \{0\}.
	\]
	Let us set ${\bf n} := (n_1, \dots , n_k)$. Precomposing this map with the map from Lemma \ref{lem: wreathConf to conf} and postcomposing with the isomorphism from Lemma \ref{lem: Wn to wreathW} we get an $\wreathOfSym{k}{\bf n}$-equivariant composition
	\begin{equation*}
			C_k(d; {\bf n})~ \longrightarrow ~\conf(\R^d, n)~ \longrightarrow ~ W_n^{\oplus d-1} \setminus \{0\} ~ \longrightarrow ~W_k(d; {\bf n}) \setminus \{0\}
	\end{equation*}
	where we consider $\wreathOfSym{k}{\bf n} \subseteq \sym_n$.
\end{proof}

\begin{remark} \label{remark: mapping through X to the sphere}
	The idea of the proof of Corollary \ref{cor: existence via operad} in the case when $n = n_1\cdots  n_k$ is not a prime power is based on equivariantly mapping the space $C_k(d; {\bf n})$ to some other space (in this case the configuration space $\conf(\R^d, n)$), from which we already know there is an equivariant map to the sphere in question.\\
	Another way to exploit this ideas was communicated to us by an anonymous referee.\\ 
	Namely, for $d \ge 3$, by the result of Avvakumov, Karasev \& Skopenkov \cite[Thm.\,2.2]{AvvakumovKarasevSkopenkov23} it follows that there is an $\sym_n$-equivariant map
	\begin{equation} \label{eq: map from X to S}
		X ~ \longrightarrow ~ S(W_n^{\oplus d-1})
	\end{equation}
	for any finite and free cellular $\sym_n$-complex $X$. Choosing, $X$ to be at least $((n-1)(d-1)-1)$-connected, there is no obstruction to the existence of an $\wreathOfSym{k}{\bf n}$-equivariant map
	\[
		\wreathOfConf{k}{d}{\bf n} ~ \longrightarrow ~ X.
	\]
	In the case when $d=2$, and $n$ is not a prime power and not twice the prime power, by using the result of Avvakumov \& Kudrya \cite[Thm.\,1.1]{AvvakumovKudrya21}, one deduces the existence of a map \eqref{eq: map from X to S}, and concludes the proof in analogous fashion.\\
	In a similar way, one may use the result of \"Ozaydin \cite[Thm.\,4.2]{Ozaydin87}. In the next section, we expand \"Ozaydin's proof in detail.
\end{remark}

\medskip
\section{Existence in Theorem \ref{theorem: main}: the \"Ozaydin trick} \label{section: existence via ozaydin}
\medskip

In this section we give yet another short proof of the existence of an $\wreathOfSym{k}{n_1, \dots , n_k}$-equivariant map 
\begin{equation} \label{eq: cellWreath to sphere, ozaydin}
	\cellWreathOfConf{k}{d}{n_1, \dots , n_k} ~ \longrightarrow ~ S(\wreathOfW{k}{d-1}{n_1, \dots , n_k})
\end{equation}
when $n_1, \dots , n_k$ are not all powers of the same prime number. 
We rely on the equivariant obstruction theory, as presented by tom Dieck \cite[Sec.\,II.3]{tomDieck:TransformationGroups}, and appropriate a trick developed by \"Ozaydin in \cite{Ozaydin87}.\\

Let ${\bf n} := (n_1, \dots ,n_k)$ and $M_k := (d-1)(n_1 \cdots n_k -1)$ as before.
Not that
\begin{compactitem}[\quad --]
	\item $\cellWreathOfConf{k}{d}{\bf n}$ is a free $\wreathOfSym{k}{\bf n}$-cell complex of dimension $M_k$, and
	
	\item the $\wreathOfSym{k}{\bf n}$-sphere $S(\wreathOfW{k}{d-1}{\bf n})$ is $(M_k-1)$-simple and $(M_k-2)$-connected.
\end{compactitem}
Consequently,  an $\wreathOfSym{k}{\bf n}$-equivariant map \eqref{eq: cellWreath to sphere, ozaydin} exists if and only if the primary obstruction
\begin{equation*}
	\mathfrak{o}_{\wreathOfSym{k}{\bf n}} \in H^{M_k}_{\wreathOfSym{k}{\bf n}}(\cellWreathOfConf{k}{d}{\bf n}; \ZZ_k)
\end{equation*}
vanishes.
Here, as before, $\ZZ_k := \ZZ_k(d-1, {\bf n}) := \pi_{M_k-1}(S(\wreathOfW{k}{d-1}{\bf n}))$ denotes the coefficient module.

\medskip
In same way, for any subgroup $G \subseteq \wreathOfSym{k}{\bf n}$, the existence of a $G$-equivariant map
\begin{equation} \label{eq: cellWreath to sphere, ozaydin}
	\cellWreathOfConf{k}{d}{n_1, \dots , n_k} ~ \longrightarrow ~ S(\wreathOfW{k}{d-1}{n_1, \dots , n_k})
\end{equation}
is equivalent to the vanishing of the obstruction class
\begin{equation*}
	\mathfrak{o}_{G} \in H^{M_k}_{G}(\cellWreathOfConf{k}{d}{\bf n}; \ZZ_k),
\end{equation*}
which in particular is the restriction of the class $\mathfrak{o}_{\wreathOfSym{k}{\bf n}}$.

\medskip
For each integer $n \ge 2$ and a prime number $p$, let us fix a $p$-Sylow subgroup $\sym_n^{(p)} \subseteq \sym_n$. For more details on Sylow subgroups, see for example \cite[Sec.\,4.5]{DummitFoote04}.

\medskip
\begin{definition}
	Let $k \ge 1$ and $n_1, \dots, n_k \ge 2$ be integers. For a prime number $p$, let us define a subgroup $\wreathOfSym{k}{n_1, \dots, n_k}^{(p)} \subseteq \wreathOfSym{k}{n_1, \dots, n_k}$ inductively as follows.
	\begin{enumerate} [(i)]
		\item For $k = 1$ we set $\wreathOfSym{1}{n} := \sym_n^{(p)}$. 
		
		\item Assume $k \ge 2$. Then, we define
		\[
			\wreathOfSym{k}{n_1, \dots, n_k}^{(p)} := (\wreathOfSym{k-1}{n_1, \dots, n_{k-1}}^{(p)})^{\times n_k} \rtimes \sym_{n_k}^{(p)} \subseteq \wreathOfSym{k-1}{n_1, \dots, n_{k-1}}^{\times n_k} \rtimes \sym_{n_k}.
		\]
	\end{enumerate}
\end{definition}

\medskip
Notice that $\wreathOfSym{k}{\bf n}^{(p)} \subseteq \wreathOfSym{k}{\bf n}$ is indeed a $p$-Sylow subgroup. 
For $k=1$ this follows from definition, while for $k \ge 2$ the recursive formula
\[
	\big|\wreathOfSym{k}{\bf n}^{(p)}\big| = \big|\wreathOfSym{k-1}{{\bf n}'}^{(p)}\big|^{n_k} \cdot \big|\sym_{n_k}^{(p)}\big|
\]
implies that the order $|\wreathOfSym{k}{\bf n}^{(p)}|$ is the maximal power of $p$ which divides the order of $|\wreathOfSym{k}{n_1, \dots, n_k}|$.

\medskip
\begin{lemma} \label{lem: fixed point of wreathSym^(p)}
	Let $d \ge 1$, $k \ge 1$ and $n_1, \dots , n_k \ge 2$ be integers and let $p$ be a prime number. 
	If $n_1, \dots , n_k$ are not all powers of $p$, then the   obstruction class 
	\[
		\mathfrak{o}_{\wreathOfSym{k}{\bf n}^{(p)}}  \in H^{M_k}_{\wreathOfSym{k}{\bf n}^{(p)}}(\cellWreathOfConf{k}{d}{\bf n}; \ZZ_k)
	\]
	vanishes.
\end{lemma}
\begin{proof}
	Assume $n_1, \dots , n_k$ are not all powers of a prime $p$.	
	Let us show that the action of the $p$-Sylow subgroup $\wreathOfSym{k}{\bf n}^{(p)} \subseteq \wreathOfSym{k}{\bf n}$ on the sphere $S(\wreathOfW{k}{d-1}{\bf n})$ has a fixed point. 
	This would produce a (constant) $\wreathOfSym{k}{\bf n}^{(p)}$-equivariant map
	\[
		\cellWreathOfConf{k}{d}{\bf n} \longrightarrow  S(\wreathOfW{k}{d-1}{\bf n}).
	\]
	Hence, from the discussion at the beginning of this section, we would have that $\mathfrak{o}_{\wreathOfSym{k}{\bf n}^{(p)}} = 0$.
	
	\medskip
	Since 
	\[
		S(\wreathOfW{k}{d-1}{\bf n}) \subseteq \wreathOfW{k}{d-1}{\bf n} \setminus \{0\}
	\]
	is an $\wreathOfSym{k}{\bf n}$-equivariant deformation retract, it is enough to show that $\wreathOfSym{k}{\bf n}^{(p)} \subseteq \wreathOfSym{k}{\bf n}$ has a nonzero fixed point in the vector space $\wreathOfW{k}{d-1}{\bf n}$. We prove this fact by induction on $k \ge 1$.
	\begin{compactenum} [\quad \rm (i)]
		\item If $k=1$, then $n := n_1$ is not a power of $p$. Thus, the $p$-Sylow subgroup $\sym_n^{(p)} \subseteq \sym_n$ does not act transitively on the set $\{1, \dots , n\}$. Consequently, $\sym_n^{(p)}$ fixes some nonzero vector $w \in W_n$, and hence it fixes $(w, \dots, w) \in W_n^{\oplus d-1}$ as well.
		
		\item Assume $k \ge 2$. We distinguish two cases. \begin{compactenum} [\rm (a)]
			\item The number $n_k$ is not a power of $p$. Then by (i) we know that $\sym_{n_k}^{(p)}$ fixes a nonzero point $v \in W_{n_k}^{\oplus d-1}$, so the nonzero point
			\[
				(0, \dots ,0; v) \in \wreathOfW{k-1}{d-1}{{\bf n}'}^{\oplus n_k} \oplus W_{n_k}^{\oplus d-1} = \wreathOfW{k}{d-1}{\bf n}
			\]
			is fixed by $\wreathOfSym{k}{\bf n}^{(p)}$.
			\item The numbers $n_1, \dots ,n_{k-1}$ are not all powers of $p$. Then, by induction hypothesis, there is a nonzero vector $V \in \wreathOfW{k-1}{d-1}{{\bf n}'}$ fixed by $\wreathOfSym{k-1}{{\bf n}'}^{(p)}$, so the nonzero vector
			\[
				(V, \dots , V; 0) \in (\wreathOfW{k-1}{d-1}{{\bf n}'})^{\oplus n_k} \oplus W_{n_k}^{\oplus d-1} = \wreathOfW{k}{d-1}{\bf n}
			\]
			is fixed by $\wreathOfSym{k}{\bf n}^{(p)}$.
		\end{compactenum} 
	\end{compactenum}
	This completes the proof of the lemma. 
\end{proof}

\medskip
Now, we give the third proof of the existence part of Theorem \ref{theorem: main}.

\medskip
\begin{corollary}[Existence in Theorem \ref{theorem: main}]
	Let $d \ge 1$, $k \ge 1$ and $n_1, \dots , n_k \ge 2$ be integers. 
	Assume that $n_1, \dots ,n_k$ are not all powers of the same prime number. Then, there exists an $\wreathOfSym{k}{n_1, \dots , n_k}$-equivariant map 
	\[
		\wreathOfConf{k}{d}{n_1, \dots , n_k} ~ \longrightarrow ~ S(\wreathOfW{k}{d-1}{n_1, \dots , n_k}).
	\]
\end{corollary}
\begin{proof}
	As already noted in the beginning of the section, the existence of an equivariant map is equivalent to vanishing of the obstruction class
	\[
		\mathfrak{o}_{\wreathOfSym{k}{\bf n}}  \in H^{M_k}_{\wreathOfSym{k}{\bf n}}(\cellWreathOfConf{k}{d}{\bf n}; \ZZ_k).
	\]
	Let $p$ be a fixed prime number dividing the order $|\wreathOfSym{k}{\bf n} |$. Since $n_1, \dots ,n_k$ are not all powers of the same prime number by Lemma \ref{lem: fixed point of wreathSym^(p)} we have
\begin{equation*}
	\mathfrak{o}_{\wreathOfSym{k}{\bf n}^{(p)}} = 0 \in H^{M_k}_{\wreathOfSym{k}{\bf n}^{(p)}}(\cellWreathOfConf{k}{d}{\bf n}; \ZZ_k).
\end{equation*}
Next, the restriction morphism
\[
	\res\colon H^{M_k}_{\wreathOfSym{k}{\bf n}}(\cellWreathOfConf{k}{d}{\bf n}; \ZZ_k) \longrightarrow H^{M_k}_{\wreathOfSym{k}{\bf n}^{(p)}}(\cellWreathOfConf{k}{d}{\bf n}; \ZZ_k),
\]
and the transfer morphism
\[
	\trf\colon  H^{M_k}_{\wreathOfSym{k}{\bf n}^{(p)}}(\cellWreathOfConf{k}{d}{\bf n}; \ZZ_k) \longrightarrow H^{M_k}_{\wreathOfSym{k}{\bf n}}(\cellWreathOfConf{k}{d}{\bf n}; \ZZ_k)
\]
satisfy 
\[
	\trf \circ \res = [\wreathOfSym{k}{\bf n}:\wreathOfSym{k}{\bf n}^{(p)}] \cdot \id
\]
due to \cite[Lem.\,5.4]{BlagojevicLueckZiegler15}. Moreover, $\res$ maps the obstruction class $\mathfrak{o}_{\wreathOfSym{k}{\bf n}}$ to the restricted obstruction class $\mathfrak{o}_{\wreathOfSym{k}{\bf n}^{(p)}}$, that is
\[
	\res(\mathfrak{o}_{\wreathOfSym{k}{\bf n}}) = \mathfrak{o}_{\wreathOfSym{k}{\bf n}^{(p)}}.
\]
Hence, it follows that
\begin{equation} \label{eq: [wr:wr^(p)] obstr = 0}
	[\wreathOfSym{k}{\bf n}:\wreathOfSym{k}{\bf n}^{(p)}] \cdot \mathfrak{o}_{\wreathOfSym{k}{\bf n}} = (\trf \circ \res) (\mathfrak{o}_{\wreathOfSym{k}{\bf n}}) = \trf(\mathfrak{o}_{\wreathOfSym{k}{\bf n}^{(p)}}) = 0.
\end{equation}
Since we have
\[
	\gcd \big\{[\wreathOfSym{k}{\bf n}:\wreathOfSym{k}{\bf n}^{(p)}]:~ p \text{ is a prime dividing } |\wreathOfSym{k}{\bf n}|\big\} = 1,
\]
there exists a $\Z$-linear combination
\[
	\sum_{p \big| |\wreathOfSym{k}{\bf n}| \text{ prime}} x_p \cdot [\wreathOfSym{k}{\bf n}:\wreathOfSym{k}{\bf n}^{(p)}] ~ = ~ 1.
\]
Finally, this equality and \eqref{eq: [wr:wr^(p)] obstr = 0} imply
\[
	\mathfrak{o}_{\wreathOfSym{k}{\bf n}} = \sum_{p \big| |\wreathOfSym{k}{\bf n}| \text{ prime}} x_p \cdot [\wreathOfSym{k}{\bf n}:\wreathOfSym{k}{\bf n}^{(p)}] \cdot \mathfrak{o}_{\wreathOfSym{k}{\bf n}} ~ = ~ 0,
\]
as desired.
\end{proof}


\medskip
\section{Continuity of partitions} 
\label{section: continuity of partitions}
\medskip

The main result of the section is Theorem \ref{theorem: continuity of partitions} on continuity of partitions. Here, we offer a direct proof. See also \cite[Thm.~5.2]{AkopyanAvvakumovKarasev18}.

\medskip
The set $\K^d$ of convex bodies in $\R^d$ with non-empty interior is a metric space with \emph{Hausdorff metric} given by
\[
	\dH(X,Y) := \max \{\sup_{x \in X}\dist(x,Y),\, \sup_{y \in Y}\dist(y,X)\}.
\]
Another metric on the same space, namely \emph{the symmetric difference metric} 
\[
	\dS(X,Y) := \mathcal{L}^d(X \triangle Y),
\]
where $\mathcal{L}^d$ is the Lebesgue measure on $\R^d$, induces the same topology on $\K^d$ (see \cite{GruberKenderov82}).

\subsection{Existence of weights}

We start with the definition of the generalised Voronoi diagram.

\medskip
\begin{definition}
	For $x := (x_1, \dots , x_n) \in \conf(\R^d, n)$ and $w := (w_1, \dots , w_n) \in \R^n$ let us denote by 
	\[
		C(x,w) := (C_1(x,w), \dots , C_n(x,w))
	\]
	the {\em generalised Voronoi diagram with sites $x$ and weights $w$}, where
	\begin{equation} \label{eq: generalised voronoi diagram, ith component}
		C_i(x, w) := \{p \in \R^d: ~ \|p-x_i\|^2 - w_i \le \|p-x_j\|^2 - w_j ~~ \text{ for all } 1 \le j \le n\}
	\end{equation}
	denotes the {\em $i$the Voronoi cell} for each $1 \le i \le n$.
\end{definition}

\medskip
For each $1 \le i \le n$, the region $C_i(x,w)$ can be considered as the set of points $p \in \R^d$ where the affine function 
\[
	f_i(p) := \|p - x_i\|^2 - \|p \|^2 - w_i
\]
takes minimal value among all functions $f_1, \dots, f_n$. In particular, regions $C_i(x,w)$ have disjoint interiors and cover the whole $\R^d$.

\medskip
Notice that $C(x,w) = C(x, w + (t, \dots ,t))$ for any $t \in \R$, so we will usually restrict our attention to weights $w \in W_n  \subseteq \R^n$. 

\medskip
\begin{definition} \label{def: K(x,w)}
	For a convex body $K \in \K^d$, sites $x \in \conf(\R^d, n)$ and weights $w \in W_n$, let us denote by
	\[
		K(x,w) := (K_1(x,w), \dots , K_n(x,w))
	\]
	the partition of $K$ given by $K_i(x,w) := K \cap C_i(x,w)$ for each $1 \le i \le n$.
\end{definition}

\medskip
As before, for $K \in \K^d$ and a probability measure $\mu$ on $\R^d$ which is absolutely continuous with respect to the Lebesgue measure, let $\EMP_{\mu}(K, n) \subseteq (\K^d)^{\times n}$ denote the space of equal mass partitions of $K$ with respect to the measure $\mu$.

\medskip
\begin{lemma} [Existence of weights, {\cite[Thm.\,1]{GKPR12}}] \label{lemma: existence of unique weights}
	Let $d \ge 1$ and $n \ge 1$ be integers, $\lambda \in (0,1]^n$ be a vector with $\lambda_1 + \dots + \lambda_n = 1$. Let $K \in \K^d$ be a convex body, and let $\mu$ be a probability measure on $\R^d$ which is absolutely continuous with respect to the Lebesgue measure. For any choice of sites $x \in \conf(\R^d, n)$ there exist unique weights $w_{K;\lambda}(x) \in W_n$ such that the generalised Voronoi diagram $C(x, w_{K;\lambda}(x))$ produces a partition of $K$
	\[
		K(x, w_{K;\lambda}(x)) = K \cap C(x, w_{K;\lambda}(x)) \in (\K^d)^{\times n}
	\]
	such that
	\[
		\mu(K_i(x, w_{K;\lambda}(x))) = \lambda_i = \lambda_i \cdot \mu(K)
	\]
	for each $1 \le i \le n$. 
	
	\medskip
	In particular, for the unique weight vector $w_K(x) := w_{K; (\frac{1}{n}, \dots , \frac{1}{n})}(x)$ we have	
	\[
		K(x) := K(x, w_K(x)) \in \EMP_{\mu}(K, n).
	\]
\end{lemma}

\medskip
The relevant case for this paper is  when $\lambda = (\frac{1}{n}, \dots , \frac{1}{n})$. 
Then, the equal mass partition $K(x) = (K_1(x), \dots , K_n(x))$ of a convex body $K$ with respect to $\mu$ guaranteed by the previous lemma is called the {\em regular equipartition} of $K$ with sites $x = (x_1, \dots , x_n)$. 

\medskip
In \cite[Sec.\,2]{KarasevHubardAronov14} and \cite[Sec.\,2]{BlagojevicZiegler15} the existence of unique weights from the lemma is proved via the theory of optimal transport. 
We include a topological proof of this fact developed by Moritz Firsching, which relies on some results from \cite{GKPR12}.

\medskip
Let us denote by $e_1, \dots , e_n \in \R^n$ the standard basis vectors, and  let $\Delta_{n-1} := \conv\{e_1, \dots , e_n\}$ be the standard $(n-1)$-dimensional simplex. 
The faces of $\Delta_{n-1}$ are convex sets of the form \mbox{$\conv\{e_i: i \in I\}$} for $I \subseteq [n]$. 
The first part of the proof of Lemma \ref{lemma: existence of unique weights} is the following lemma.

\medskip
\begin{lemma} \label{lem: simplex to simplex, surjective and injective on the preimage of interior}
	Let $f : \Delta_{n-1} \longrightarrow \Delta_{n-1}$ be a continuous self map of the simplex with coordinate functions $f_i$ for $i \in [n]$. Then, the following statements hold.
	\begin{compactenum}[\quad \em (i)]
		\item A map $f$ is surjective if $f(\sigma) \subseteq \sigma$ holds for all faces $\sigma \subseteq \Delta_{n-1}$.
		
		\item A map $f$ is injective on the preimage $f^{-1}(\interior(\Delta_{n-1}))$ of the interior if additionaly the following condition holds:
			\begin{compactitem}[ ~]
				\item For a non-empty subset $I \subsetneq [n]$, a point $x \in \interior(\Delta_{n-1})$ and $\alpha \in (0,1)$,  let us denote   
					\[
						t := \sum_{i \in I} x_i \in (0,1) ~~~ \text{ and } ~~~  x^I_{\alpha} := \alpha\left(\sum_{i \in I} \frac{x_i}{t} e_i\right) + (1 - \alpha) \left(\sum_{j \in [n] \setminus I} \frac{x_j}{1-t} e_j\right) \in \interior (\Delta_{n-1}).
					\]
					Then, for every $x \in \interior(\Delta_{n-1})$, every non-empty subset $I \subsetneq [n]$ and every $\alpha \in (t, 1)$ we have $f_i(x^I_{\alpha}) \ge f_i(x)$ for every $i \in I$, with strict inequality holding for at least one $i\in I$.
			\end{compactitem}
	\end{compactenum} 
\end{lemma}

\medskip
The condition in part (ii) of the lemma says the following. For a point $x \in \interior(\Delta_{n-1})$ we have that $\sum_{i \in I} \frac{x_i}{t} e_i$ is its projection to the face of the simplex spanned by $\{e_i: i\in I\}$, while $\sum_{j \in [n] \setminus I} \frac{x_j}{1-t} e_j$ is its projection to the complementary face. In particular, $x$ is a convex combination of the two projections, namely $x = x^I_t$. Notice that 
\[
	\{x^I_{\alpha}: \alpha \in (t, 1)\} = \Big(x, ~\sum_{i \in I} \frac{x_i}{t} e_i\Big ) \subseteq \interior(\Delta_{n-1}),
\]
and the condition requires that for each point $x^I_{\alpha}$ in this segment the values of $f_i$, for $i \in I$, are at least as large as $f_i(x)$, with at least one strict inequality.

\begin{proof}[Proof of Lemma \ref{lem: simplex to simplex, surjective and injective on the preimage of interior} due to Moritz Firsching]
(i) The map $f$ is homotopic to the identity via $f_{\tau} :=\tau f+(1-\tau)\id_{\Delta_{n-1}}$ for $\tau \in [0,1]$. Each map $f_{\tau}$ in the homotopy satisfies $f_\tau (\sigma) \subseteq \sigma$ and hence induces a homotopy of quotient maps
		\[
			\widetilde{f} \simeq \id_{\Delta_{n-1}/\partial \Delta_{n-1}}: \Delta_{n-1}/\partial \Delta_{n-1} \longrightarrow \Delta_{n-1}/\partial \Delta_{n-1}.
		\]
		Since $\Delta_{n-1}/\partial \Delta_{n-1} \approx S^{n-1}$, the $\widetilde{f}$ is not nullhomotopic and hence is surjective, since every non-surjective self map of the sphere is necessarily nullhomotopic. The surjectivity of the quotient map $\widetilde{f}$ implies the surjectivity of $f$ by the fact that $\interior(\Delta_{n-1}) = \Delta_{n-1} \setminus \partial \Delta_{n-1}$ is dense in $\Delta_{n-1}$.
		
\medskip\noindent
(ii) The injectivity argument is similar to the one in \cite[Thm.\,1]{GKPR12}. Suppose for two different points $x,y \in \Delta_{n-1}$ we have $f(x) = f(y) \in \interior(\Delta_{n-1})$. Then $x,y \in \interior(\Delta_{n-1})$, since $f$ maps each face to iteself. Define
			\[
				I(x,y) := \{ i \in [n]:~ x_i/y_i= \min_{1 \le j \le n} x_j/y_j\}.
			\] 
	Since $x \neq y$ we have $\emptyset \neq I \subsetneq [n]$. We will inductively define a sequence of points
			\[
				x = x^0, x^1, \dots , x^k  = y \in \interior(\Delta_{n-1})
			\]
	for some integer $k \ge 1$, which satisfy 
			\[
				I ({x^0, y}) \subsetneq I({x^1, y}) \subsetneq \dots \subsetneq I({x^{k-1}, y}) \subsetneq [n]
			\]
	and such that for each $1 \le l \le k$ we have $f_i(x^{l-1}) \le f_i(x^{l})$ for all $i \in I({x^{l-1}, y})$, with the strict inequality holding true for at least one index in $i \in I({x^{l-1}, y})$. Indeed, if we prove this, than for the index $i \in I(x^0, y)$ for which the strict inequality $f_i(x^0) < f_i(x^1)$ holds, we would have
			\[
				f_i(x) = f_i(x^0) < f_i(x^1) \le \dots \le f_i(x^k) = f_i(y),
			\]
	which contradicts $f(x) = f(y)$.
				
	\medskip
	Suppose $x^0, \dots , x^{l-1} \neq y$ with the above properties have already been constructed for some $l \ge 1$, and let us construct $x^l$. Since $I({x^{l-1}, y}) \subsetneq [n]$, by the assumption of part (ii), there exists some $\alpha$ with 
			\[
				\sum_{i \in I({x^{l-1}, y})} x^{l-1}_i < \alpha < 1
			\]
	such that the point $x^l := x_{\alpha}^{I({x^{l-1}, y})}$ satisfies $I({x^{l-1}, y}) \subsetneq I({x^l, y})$ and $f_i(x^{l}) \le f_i(x^{l})$ for all $i \in I({x^{l-1}, y})$, with the strict inequality for at least one index $i \in I({x^{l-1}, y})$. If $I({x^{l}, y}) = [n]$, then $k := l$ and $x^l = y$, so we stop induction. Otherwise, we continue for finite more steps until $I({x^k, y}) = [n]$.
\end{proof}

\begin{proof} [Proof of Lemma \ref{lemma: existence of unique weights} due to Moritz Firsching]
	Let us fix a point $x \in \conf(\R^d, n)$. We want to show the existence of a weight vector $w_{K; \lambda}(x) \in W_n$.
	
	\medskip
	 Given a point $z = (z_1, \dots ,z_n) \in \Delta_{n-1}$ set 
	 $
	 	w(z) := (\log(z_1),...,\log(z_n))
	 $.
	 Here we allow $\log(0) := -\infty$ and extend the definition of $C(x,w)$ to include weight vectors $w$ with some (but not all) coordinates being $-\infty$. 
	 We define the continuous map $m$ by
	 \[
	 	m\colon  \Delta_{n-1} \longrightarrow \Delta_{n-1}, \qquad z \longmapsto m(z) := (\mu(K_1(x,w(z))), \dots , \mu(K_n(x,w(z)))).
	 \]
If this map is surjective and injective on the preimage $m^{-1}(\interior(\Delta_{n-1}))$, we can set $w_{K; \lambda}(x) := w(z)$, where $\{z\} = m^{-1}(\{\lambda\})$ is the unique point in the fiber of $\lambda \in \interior(\Delta_{n-1})$. To prove the required properties of the map $m$ we use Lemma \ref{lem: simplex to simplex, surjective and injective on the preimage of interior}. 

\medskip
Before doing that, let us see why the map $m$ is continuous. 
First, notice that if we show it is continuous on the interior $\interior(\Delta_{n-1})$, the continuity of $m$ on the whole domain follows by density of the interior and the fact that the value of $m$ on the boundary is the limit value of interior points. 
Thus, let $z^k \to z$ as $k \to \infty$ be a convergence in $\interior(\Delta_{n-1})$. 
Then by part (ii) of Lemma \ref{lemma: hyperplane cuts off continuously} for each $1 \le i \le n$ we have
\[
	\leb\Big(K_i(x,w(z^k)) \triangle K_i(x,w(z))\Big) \xrightarrow{k \to \infty} 0.
\]
By \eqref{eq: lebesgue -> 0 => mu -> 0} it further follows that for each $1 \le i \le n$ we have
\[
	\mu(K_i(x,w(z^k))) \xrightarrow{k \to \infty} \mu(K_i(x,w(z))),
\]
proving ultimately that $m(z^k) \to m(z)$ as $k \to \infty$.

\medskip
To prove surjectivity of $m$, by Lemma \ref{lem: simplex to simplex, surjective and injective on the preimage of interior} part (i) it is enough to show that $m$ maps each face to itself. Indeed, assume for $z \in \Delta_{n-1}$ we have $z_i = 0$ for some $1 \le i \le n$. Then the $i$th coordinate of $w(z)$ is $-\infty$ and the $i$th coordinate of $m(z)$ is zero, since the $i$th Voronoi region $C_i(x, w(z))$ is empty. 
Therefore a face of $\Delta_{n-1}$ is mapped to itself by $m$.

\medskip
For injectivity of $m$ on the preimage $m^{-1}(\interior(\Delta_{n-1}))$ we use Lemma \ref{lem: simplex to simplex, surjective and injective on the preimage of interior} part (ii). We want to show that for any point $z \in \Delta_{n-1}$, nonempty $I\subsetneq [n]$, $t := \sum_{i \in I} z_i$ and $\alpha \in (t, 1)$ the condition in part (ii) is satisfied. Notice that the coordinate function of $w$ satisfies 
\[
	w_i(z_{\alpha}^I) = \log(\alpha/t) + w_i(z), ~~ \text{ for } i \in I
\]
and 
\[
	w_j(z_{\alpha}^I) = \log((1-\alpha)/(1-t)) + w_j(z), ~~ \text{ for } j \in [n] \setminus I.
\]
If we denote by $w'$ the weight vector obtained from $w(z_{\alpha}^I)$ by adding a constant 
\[
	-\log((1-\alpha)/(1-t)) = \log((1-t)/(1-\alpha))
\]
to each coordinate, we have equality of Voronoi diagrams
\[
	C(x, w(z_{\alpha}^I)) = C(x, w'),
\]
so one can use the weight vector $w'$ to calculate $m(z_{\alpha}^I)$.
Moreover, we have
\[
	w'_i = \log(\alpha/t) + w_i(z) + \log((1-t)/(1-\alpha)) = \log(1/t - 1) - \log(1/\alpha - 1) + w_i(z) > w_i(z)
\]
for $i \in I$ and $w'_j = w_j(z)$ for $j \in [n] \setminus I$. In particular, for each $i \in I$ we have 
\[
	C_i(x, w(z)) \subseteq C_i(x, w') = C_i(x, w(z_{\alpha}^I)),
\]
which implies $m_i(z_{\alpha}^I) \ge m_i(z)$. Therefore, by \cite[Lem.\,2]{GKPR12} it follows that there exists an index $i \in I$ such that $m_i(z_{\alpha}^I) > m_i(z)$. Hence, the assumption in part (ii) of Lemma \ref{lem: simplex to simplex, surjective and injective on the preimage of interior} holds. 
\end{proof}

\subsection{Continuity of partitions}

Lemma \ref{lemma: existence of unique weights} implies that, for a given $K \in \K^d$, the function
\[
	\conf(\R^d, n) \longrightarrow W_n, ~~~ x \longmapsto w_K(x)
\]
exists and is $\sym_n$-equivariant. In \cite[Lem.\,3]{HubardAronov10} it was shown that the function is continuous on parameter $x \in \conf(\R^d, n)$. 
Since the function also depends on the convex body $K \in \K^d$ it is natural to ask: {\em Is the fuctnion $w_K(x)$ continuous in parameters $(x, K) \in \conf(\R^d, n) \times \K^d$?}

\medskip
We give a positive answer to this question in the lemma which follows. 
The proof is postponed to the end of Section \ref{subsection: deterred lemmas}. 

\medskip
\begin{lemma} [Continuity of weights] \label{lemma: continuity of weights}
	Let $\mu$ be a probability measure on $\R^d$ which is absolutely continuous with respect to the Lebesgue measure. The assignment
	\[
		\conf(\R^d, n) \times \K^d \longrightarrow W_n, ~ (x, K) \longmapsto w_K(x),
	\]
	given by Lemma \ref{lemma: existence of unique weights}, is continuous and $\sym_n$-equivariant. The $\sym_n$-action on the $\K^d$-coordinate in the domain is assumed to be trivial.
\end{lemma}

\medskip
For a given convex body $K \in \K^d$, the continuity of the assignment 
\[
	\conf(\R^d, n) \longrightarrow \EMP_{\mu}(K, n), \qquad x \longmapsto K(x) = K(x, w_K(x)),
\]
given by Lemma \ref{lemma: existence of unique weights}, is proved in \cite[Thm.\,2.1]{KarasevHubardAronov14} and \cite[Thm.\,2.1]{BlagojevicZiegler15}. 
Again as before, this assignment depends on $K$, so the next natural question arises: {\em Is $K(x)$ continuous in parameters $(x, K) \in \conf(\R^d, n) \times \K^d$?}

\medskip
 The main result of this section is the following theorem which gives the positive answer to the previous question.

\medskip
\begin{theorem} [Continuity of partitions] \label{theorem: continuity of partitions}
	The function
	\[
		\conf(\R^d, n) \times \K^d \longrightarrow (\K^d)^{\times n}, \qquad (x, K) \longmapsto K(x),
	\]
	where $K(x)$ denotes the regular partition of $K$ with sites $x$ from Lemma \ref{lemma: existence of unique weights}, is a continuous $\sym_n$-equivariant map which satisfies the restriction property
	\[
		\conf(\R^d, n) \times \{ K \} \longrightarrow \EMP_{\mu}(K, n), \qquad (x, K) \longmapsto K(x),
	\]
	for each $K \in \K^d$.
\end{theorem}
\begin{proof}		
	To show the claim, it is enough to show that for each $1 \le i \le n$ the coordinate map
	\[
		\conf(\R^d, n) \times  \K^d \longrightarrow \K^d, \qquad (x, K) \longmapsto K_i(x,w_K(x)) = K \cap C_i(x, w_K(x))
	\]
	is continuous, where weights $w_K(x)$ are given by Lemma \ref{lemma: existence of unique weights} and the region $C_i(x, w_K(x))$ is the $i$th component \eqref{eq: generalised voronoi diagram, ith component} of the generalised Voronoi diagram $C(x,w_K(x))$ with sites $x$ and weights $w_K(x)$.
	 
	To show sequential continuity of the coordinate map, let
	\[
		(x^k, K^k) \xrightarrow{~~ k \to \infty ~~} (x, K) \in \conf(\R^d, n) \times \K^d
	\]
	be a converging sequence, and let us denote the weights by $w := w_K(x)$ and $w^k := w_{K^k}(x^k)$. 
	By Lemma \ref{lemma: continuity of weights}, we have $w^k \to w$ as $k \to \infty$. 
	
	\medskip
	Moreover, since $K_i(x,w) = K \cap C_i(x,w)$ has a non-empty interior and $(x^k, w^k) \to (x,w)$ as $k \to \infty$, it follows that 
	\[
		K_i(x^k,w^k) = K \cap C_i(x^k,w^k)
	\]
	has a non-empty interior as well, hence $K_i(x^k,w^k) \in \K^d$. 
	Therefore, all notions used in the following string of inequalities are well defined. We have
	\begin{align*}
		\dS(K_i(x, w), K^k_i(x^k, w^k)) &\le \dS(K_i(x, w), K_i(x^k, w^k)) + \dS(K_i(x^k, w^k), K^k_i(x^k, w^k))\\
		& \le \dS(K_i(x, w), K_i(x^k, w^k)) + \dS(K, K^k) \xrightarrow{ ~~ k \to \infty ~~} 0,
	\end{align*}
	because the first summand in the last row tends to zero as $k \to \infty$ by Lemma \ref{lemma: hyperplane cuts off continuously} (ii), so the claim of the lemma follows.
\end{proof}

\subsection{Auxiliar lemmas} \label{subsection: deterred lemmas}

For a vector $(a,b) \in \R^d \times \R$ such that $a \neq 0 \in \R^d$ let us denote by
\[
	H_{(a,b)} := \{p \in \R^d \ : \ \langle a,  p\rangle + b = 0\}
\]
the affine hyperplane induced by $(a,b)$, and by 
\[
	H^+_{(a,b)} := \{p \in \R^d\ : \ \langle a,  p\rangle + b \ge 0\} 
\]
the corresponding closed half-space.

\medskip
\begin{lemma} \label{lemma: hyperplane cuts off continuously}
	Let $K \in \K^d$. Then, the following statements are true.
	\begin{compactenum}[\quad\rm   (i)]
		\item Assume the convergence
		\[
			(a^k, b^k) \xrightarrow{~k \to \infty ~} (a,b) \in \R^{d}\times \R
		\]
		with $a, a^k\neq 0$, as well as
		\[
			K \cap H_{(a,b)}^{+},~ K \cap H_{(a^k,b^k)}^{+} \in \K^d
		\]
		for each $k \ge 1$. Then
	\[
		K \cap H_{(a^k,b^k)}^+ \xrightarrow{~\dS~} K \cap H_{(a,b)}^+
	\]
	as $k \to \infty$.
	\item Assume the convergence
	\[
		(x^k, w^k) \xrightarrow{~ k \to \infty ~} (x,w) \in \conf(\R^d, n) \times W_n,
	\]
	and let $1 \le i \le n$, as well as
	\[
		K_i(x,w), K_i(x^k, w^k) \in \K^d
	\]
	for each $k \ge 1$. Then
	\[
		K_i(x^k,w^k) \xrightarrow{~\dS~} K_i(x,w)
	\]
	as $k \to \infty$.
	\end{compactenum}
	\end{lemma}

\begin{proof}
(i)
Since $\dS$ and $\dH$ induce the same topology on $\K$, it is enough to prove convergence in $\dH$ metric. 
Recall that
		\[
			\dH(K\cap H_{(a^k,b^k)}^+, K\cap H_{(a,b)}^+) = \max \big\{\sup_{x \in K \cap H_{(a^k,b^k)}^+}\dist(x, K\cap H_{(a,b)}^+), \sup_{x \in K \cap H_{(a,b)}^+}\dist(x, K\cap H_{(a^k,b^k)}^+)\big\}.
		\]
	
For the first supremum we have 
		\begin{equation*} \label{eq: first sup}
			\sup_{x \in K \cap H_{(a^k,b^k)}^+}\dist(x, K\cap H_{(a,b)}^+) \xrightarrow{~ k \to \infty ~} 0	.
		\end{equation*}
Indeed, assume to the contrary that for some sequence $(z_k \in K \cap H_{(a^k,b^k)}^+)_{k \ge 1}$, we have that 
		\[
			\dist(z_k, K \cap H_{(a,b)}^+) > 2\varepsilon.
		\]
Since $K$ is compact, we could have chosen the sequence $z_k$ from the very start such that $z_k \xrightarrow{~ k \to \infty ~} z \in K$, and so
		\[
			\dist(z, K \cap H_{(a,b)}^+) > \varepsilon.
		\]
From this we obtain a contradiction by showing $z_k \notin H_{(a^k, b^k)}^+$ for $k >> 0$. This is indeed true, since
		\[
			\langle z_k , a^k \rangle + b^k \xrightarrow{~ k \to \infty ~} \langle z,  a\rangle + b < 0.
		\]
		
\medskip
For the second supremum we have that
		\begin{equation*} \label{eq: second sup}
			\sup_{x \in K \cap H_{(a,b)}^+}\dist(x, K\cap H_{(a^k, b^k)}^+) \xrightarrow{~ k \to \infty ~} 0	.
		\end{equation*}
Indeed, assume to the contrary that for some sequence $(z_k \in K \cap H_{(a,b)}^+)_{k \ge 1}$ we have
		\[
			\dist(z_k, K \cap H_{(a^k, b^k)}^+) > 2\varepsilon.
		\]
Since $K$ is compact, there is a subsequence with $z_k \xrightarrow{~ k \to \infty ~} z \in K \cap H_{(a,b)}^+$, hence
		\[
			\dist(z, K \cap H_{(a^k, b^k)}^+) > \varepsilon.
		\] 
for each $k \ge 1$. Notice first that $z \in K \cap H_{(a,b)}$. Indeed, strict inequality $z \cdot a + b > 0$ would 		
		\[
			\langle z, a^k\rangle + b^k \xrightarrow{~ k \to \infty ~} \langle z , a\rangle + b > 0,
		\] 
which would mean $z \in K \cap H_{(a^k, b^k)}^+$, which is impossible. Since $(H^+_{(a, b)} \setminus H_{(a, b)}) \cap K \neq \emptyset$ by assumption, by convexity there exists 
		\[
			y \in (H^+_{(a, b)} \setminus H_{(a, b)}) \cap K
		\]
with $\dist(x,y) < \varepsilon$. Moreover, we have $y \in K \cap H_{(a^k, b^k)}^+$ for $k >> 0$, since 
		\[
			\langle y, a^k\rangle + b^k \xrightarrow{~ k \to \infty ~} \langle y, a\rangle + b > 0.
		\] 
This means that 
		\[
			\varepsilon > \dist(x,y) \ge \dist(z, K\cap H_{(a^k, b^k)}^+) > \varepsilon
		\] 
for $k>>0$, which is a contradiction.

\medskip\noindent	
(ii) We will first show that for two converging sequences $A^k \xrightarrow{~\dS~} A$ and $B^k \xrightarrow{~\dS~} B$ in $\K^d$ such that $A^k \cap B^k, A\cap B \in \K^d$, we have 
	\[
		A^k \cap B^k \xrightarrow{~\dS~} A \cap B
	\]
as $k \to \infty$. Indeed, from
	\[
		(A^k \cap B^k) \triangle (A \cap B) \subseteq (A^k \triangle A) \cup (B^k \triangle B)
	\]
it follows that
	\[
		\dS(A^k \cap B^k, A \cap B) \le \dS(A^k, A) + \dS(B^k ,B) \xrightarrow {~ k \to \infty ~}0,
	\]
as desired.
	
\medskip	
Back to the proof of the main claim. For $(x,w) \in \conf(\R^d, n) \times W_n$ and $1 \le i \le n$ the $i$th component $C_i(x,w)$ of the generalised Voronoi diagram $C(x,w)$ is equal to the intersection of closed half-spaces
	\[
		H^+_{i,j}(x,w) := \{p \in \R^d: \|p-x_j\|^2-\|x_j\|^2 - w_j - \|p-x_i\|^2-\|x_i\|^2 - w_i \ge 0\}
	\]
for $1 \le j \le i$ and $j \neq i$. Therefore, by the repeated use of the above intersection argument and part (i), we have
	\[
		\dS(K_i(x,w), K_i(x^k,w^k)) = \dS(K \cap \bigcap_{j \neq i} H^+_{i,j}(x,w), K \cap \bigcap_{j \neq i} H^+_{i,j}(x^k,w^k)) \xrightarrow {~ k \to \infty ~}0,
	\]
which completes the proof.
\end{proof}

\medskip
Recall, for a measure $\mu$ which is absolutely continuous with respect to the Lebesgue measure $\leb$ on $\R^d$, by the Radon-Nikodym theorem, there exists an integrable function $f:\R^d \to \R$ such that
	\[
		\mu(A) = \int_A f ~ d\leb,
	\]
	for each measurable set $A \subseteq \R^d$. We will need the next measure-theoretic claim.

\medskip
\begin{lemma}
	Let $\mu$ be a probability measure on $\R^d$ which is absolutely continuous with respect to the Lebesgue measure $\leb$. Then, the following implication holds
	\begin{equation} \label{eq: lebesgue -> 0 => mu -> 0}
		\leb(A^k) \xrightarrow{~ k \to \infty ~} 0 ~~ \implies \mu(A^k) \xrightarrow{~ k \to \infty ~} 0,
	\end{equation}
	where $(A^k)_{k \ge 1}$ is any sequence of measurable sets in $\R^d$.
\end{lemma}
\begin{proof}
	The proof, in essence, follows from a use of the Dominant Convergence Theorem \cite[Sec.\,A3.2]{Alt2016}.
	Let $\chi_A$ denotes the characteristic function of a measurable set $A \subseteq \R^d$. 
	We prove \eqref{eq: lebesgue -> 0 => mu -> 0} in two steps.
	
	\medskip
	Let is first show that a sequence of functions $(f\cdot \chi_{A_k})_{k \ge 1}$ converges pointwise almost everywhere to the zero function. Indeed, let $g$ denote the pointwise limit of the sequence $(f\cdot \chi_{A_k})_{k \ge 1}$. Let $\varepsilon > 0$ and let us restrict to a subsequence such that
		$
			\sum_{k=1}^{\infty} \leb(A^k) < \varepsilon
		$.
		Then, we have
		\[
			\{z \in \R^d:~g(z) \neq 0\} \subseteq \bigcup_{k=1}^{\infty} A^k,
		\]
		and therefore 
		\[
			\leb(\{z \in \R^d:~g(z) \neq 0\}) \le \sum_{k=1}^{\infty} \leb(A^k) < \varepsilon.
		\]
		Since this holds for any $\varepsilon > 0$, it follows that $\leb(\{z \in \R^d:~g(z) \neq 0\}) = 0$, so $g$ is zero almost everywhere.

	\medskip
	Finally, by the Dominant Convergence Theorem applied to the dominant $f$, we have 
	\[
		\lim_{k \to \infty} \mu(A^k) = \lim_{k \to \infty} \int_{\R^d}f\cdot \chi_{A_k}(x) ~ d\leb(x)  = \int_{\R^d}\big(\lim_{k \to \infty} f\cdot \chi_{A_k}(x)\big) ~ d\leb(x) = 0,
	\]
	which finishes the proof of the implication \eqref{eq: lebesgue -> 0 => mu -> 0}.
\end{proof}

\medskip
We are now in the position to prove Lemma \ref{lemma: continuity of weights}.

\medskip
\begin{proof} [Proof of Lemma \ref{lemma: continuity of weights}]
	Let $x^k \longrightarrow x \in \conf(\R^d, n)$ and $K^k \xrightarrow{~\dH~} K \in \K^d$ as $k \to \infty$. We want to show
	\[
		w_{K^k}(x^k) \xrightarrow{~ k \to \infty ~} w_K(x) \in W_n.
	\]
	Let us shorten the notation and denote $w := w_K(x)$ and $w^k := w_{K^k}(x^k)$ for each $k \ge 1$.
	
	\medskip
	First, notice that the sequence $(w^k)_{k \ge 1}$ is bounded. Indeed, $K^k \xrightarrow{\dH} K$ as $k \to \infty$ implies
	\[
		\sup_{z\in K^k} \dist(z, K) \xrightarrow{~ k \to \infty ~} 0,
	\]
	so bodies $K^k$ are contained in a bounded region around $K \subseteq\R^d$. Similarly, $x^k \to x$ as $k \to \infty$ implies that the sites $x^k$ live in a bounded region around $x \in \conf(\R^d, n)$. Therefore, for any $1 \le i < j \le n$ the difference $|w_i^k - w_j^k|$ must be bounded for all $k \ge 1$, proving that the sequence $(w^k)_{k \ge 1}$ is bounded in $W_n$. Indeed, if for some $1 \le i < j \le n$ we have $|w_i^k - w_j^k| \to \infty$ as $k \to \infty$, it would imply 
	\[
	K^k_i(x^k, w^k) = K^k \cap C_i(x^k, w^k) = \emptyset
	\]
	for $k >> 0$, which is a contradiction.
	
	\medskip
	Convergence $w^k \to w$ as $k \to \infty$ is equivalent to the same convergence for each subsequence of $(w^k)_{k \ge 1}$. Let us assume the latter is not true and seek a contradiction. By the boundedness of weights, and after possibly restricting to a subsequence, we have $w^k \longrightarrow w' \in W_n$ as $k \to \infty$, for some $w' \neq w$. Due to the uniqueness of the weight vector in $W_n$ for given sites $x$ and a convex body $K$, contradiction would follow from the fact that
	\[
		\mu(K_i(x,w')) = \mu(K_i(x,w))
	\]
	for each $1\leq i\leq n$.
	
	\medskip
	Since $\leb(K \triangle K^k) \to 0$ as $k \to \infty$, by \eqref{eq: lebesgue -> 0 => mu -> 0} we have
	\[
		|\mu(K_i(x^k, w^k)) - \mu(K^k_i(x^k, w^k))| \le \mu(K_i(x^k, w^k)) \triangle K^k_i(x^k, w^k)) \le \mu(K \triangle K^k) \xrightarrow{~ k \to \infty ~} 0.
	\]
	From $(x^k, w^k) \longrightarrow (x, w')$ as $k \to \infty$ and Lemma \ref{lemma: hyperplane cuts off continuously} (ii) it follows that $K_i(x^k, w^k) \xrightarrow{\dS} K_i(x, w')$ as $k \to \infty$, therefore by \eqref{eq: lebesgue -> 0 => mu -> 0} we have
	\[
		|\mu(K_i(x, w')) - \mu(K_i(x^k, w^k))| \xrightarrow{~ k \to \infty ~} 0.
	\]
	Putting these two convergences together we get
	\begin{equation} \label{eq: approximation of areas}
		|\mu(K_i(x, w')) - \mu(K^k_i(x^k, w^k))| \xrightarrow{~ k \to \infty ~} 0.
	\end{equation}
	In particular, from $|\mu(K^k) - \mu(K)| \le \mu(K \triangle K^k) \to 0$ as $k \to \infty$, we get 
	\[
		\mu(K^k_i(x^k, w^k)) = \frac{1}{n}\mu(K^k) \xrightarrow{~k \to \infty~} \frac{1}{n}\mu(K) = \mu(K_i(x, w)),
	\]
	which together with \eqref{eq: approximation of areas} implies $\mu(K_i(x,w')) = \mu(K_i(x,w))$ as desired.
\end{proof}

\end{document}